\documentclass[12pt]{article}
\usepackage[colorlinks]{hyperref}
\usepackage{color}
\usepackage{graphicx}
\usepackage{graphics}
\usepackage{makeidx}
\usepackage{showidx}
\usepackage{latexsym}
\usepackage{amssymb}
\usepackage{verbatim}
\usepackage{amsmath}
\usepackage{amsthm}
\usepackage{amsfonts}
\usepackage{amssymb,amsmath}
\usepackage{latexsym,amsthm,amscd}
\usepackage[all]{xy}

\newtheorem{prop}{Proposition}[section]
\newtheorem{lemma}{Lemma}[section]
\newtheorem{defn}{Definition}[section]
\newcounter{alphthm}
\newtheorem{Lemma}[alphthm]{Lemma}

\newtheorem{cor}{Corollary}[section]

\newtheorem{rem}{Remark}[section]

\newtheorem{ex}{Example}[section]
\newtheorem{thm}{Theorem}
\newtheorem{lem}{Lemma}[section]
\newcommand{\be}{\begin{equation}}
\newcommand{\ee}{\end{equation}}
\newcommand{\ben}{\begin{enumerate}}
\newcommand{\een}{\end{enumerate}}
\newcommand{\beq}{\begin{eqnarray}}
\newcommand{\eeq}{\end{eqnarray}}
\newcommand{\beqn}{\begin{eqnarray*}}
\newcommand{\eeqn}{\end{eqnarray*}}

\newcommand{\p}{\partial}
\newcommand{\bpf}{\begin{proof}}
\newcommand{\epf}{\end{proof}}
\newcommand{\bl}{\begin{lem}}
\newcommand{\el}{\end{lem}}
\newcommand{\bp}{\begin{prop}}
\newcommand{\ep}{\end{prop}}
\newcommand{\bd}{\begin{defn}}
\newcommand{\ed}{\end{defn}}
\newcommand{\bt}{\begin{thm}}
\newcommand{\et}{\end{thm}}

\def\nn{\nonumber}
\newcommand\bpr{\begin{prop}}
\newcommand\epr{\end{prop}}

\title{Ricci flow on Finsler manifolds}
\author{B. Bidabad\thanks{The corresponding author  bidabad@aut.ac.ir; behroz.bidabad@math.univ-toulouse.fr}\, and\, M. K. Sedaghat}
\date{}
\begin{document}
\maketitle

\begin{abstract}
This paper investigates the short-time existence and uniqueness of Ricci flow solutions on Finsler manifolds. The main results of this paper are theorems demonstrating the short-time existence of the flow solution for $n$-dimensional Finsler manifolds and the uniqueness of the solution for isotropic Finsler manifolds. Two examples are also presented to illustrate the results.
 \end{abstract}
\vspace{.5cm}
{\footnotesize\textbf{AMS subject classification}: {53C60, 53B40.}
\section{Introduction}
One of the primary aims of geometric flows is to produce canonical geometric structures by deforming general initial metrics to these canonical structures. The Ricci flow theory and its applications have become a rapidly developing branch of mathematics, with its most remarkable achievement being G. Perelman's proof of Thurston's geometrization conjecture. Hamilton's Ricci flow, introduced in 1982, is a geometric flow with numerous applications in physics and real-world problems. The Ricci flow conformally deforms the Riemannian metric to its induced curvature, such that the curvature tensors evolve by a system of diffusion equations that distributes the curvature uniformly over the manifold. Hamilton demonstrated that, on a closed manifold, there is a unique solution to the Ricci flow over a sufficiently short-time, and that on a compact three-manifold with an initial metric of positive Ricci curvature, the Ricci flow converges, after re-scaling to keep constant volume, to a metric of positive constant sectional curvature, proving that the manifold is diffeomorphic to the three-sphere $\mathbb{S}^{3}$ or a quotient of the three-sphere $\mathbb{S}^{3}$ by a linear group of isometries.

The Ricci flow has demonstrated its great potential in geometry and physics, solving various problems that cannot easily be addressed by alternative methods. In discrete geometry, the discrete Ricci flow is a powerful tool for computing metrics with prescribed Gaussian curvatures on general surfaces. Examples of applications related to the discrete Ricci flow include global conformal parameterizations in computer graphics and analysis of medical imaging, surface matching, and manifold splines.

 In Finsler geometry, a natural generalization of Riemannian geometry, the problem of constructing the Finslerian Ricci flow raises several new conceptual and fundamental issues regarding the compatibility of geometrical and physical objects and their optimal configurations. A fundamental step in studying any system of evolutionary partial differential equations is to show the short-time existence and uniqueness of solutions. Several joint works have studied the evolution of a family of Finsler metrics along the Ricci flow, demonstrating that Finsler Ricci flow exists in a short time and converges to a limit metric (see \cite{YB2}). S. Lakzian has further applied the Finsler–Ricci flow to Harnack estimates of the heat equation, using the tools presented by S. Ohta and K. Sturm in \cite{OhSt}.

 The authors of the present work have shown the existence and uniqueness of solutions to the Ricci flow on Finsler surfaces in \cite{BKS1}, extending the work of R. S. Hamilton on Riemannian surfaces \cite{Ham2}. These results have numerous applications in general relativity, though too many to list here. However, in dimensions $n>2$, the uniqueness of solutions is no longer valid,  \cite{BKS3}. In \cite{BKS2}, it was further shown that the Finsler Ricci flow preserves the positivity of the reduced $hh$-curvature on finite time, and the evolution of the Ricci scalar is a parabolic-type equation. Additionally, if the initial Finsler metric has positive flag curvature, then the flag curvature, as well as the Ricci scalar, remain positive as long as the solution exists.

In the present work, we investigate the Ricci flow on closed $n$-dimensional Finsler manifolds, with $n>2$, and prove the short-time existence of solutions. Intuitively, since the Ricci flow system of equations is only weakly parabolic, the standard theory of parabolic equations cannot be used to prove its short-time existence and uniqueness. Following the procedure outlined by D. DeTurck in the Riemannian space \cite{DeT}, we introduce the Finslerian Ricci-DeTurck flow by Eq. (\ref{22}) and study the solutions. We then find a solution to the original Ricci flow by pulling back the solution of the Ricci-DeTurck flow using suitable diffeomorphisms and prove the existence of short-time solutions.

\begin{thm}\label{main8}
Let $M$ be a compact Finsler manifold. Given any initial Finsler structure $F_{0}$, there exists a real number $T>0$ and a smooth one-parameter family of Finsler structures $\tilde{F}(t)$, for all $t\in[0,T)$, such that $\tilde{F}(t)$ is a solution to the Finslerian Ricci-DeTurck flow with $\tilde{F}(0)=F_{0}$.
\end{thm}
Then, a solution to the original Ricci flow is obtained by pulling back the solution to the Ricci-DeTurck flow via appropriate diffeomorphisms. This result leads to
\begin{thm} \label{main14}
Let $M$ be a compact Finsler manifold. Given any initial Finsler structure $F_{0}$, there exists a real number $T>0$ and a smooth one-parameter family of Finsler structures $F(t)$, for all $t\in[0,T)$ such that $F(t)$ is a solution to the Finslerian Ricci flow and $F(0)=F_{0}$.
\end{thm}
The uniqueness of the solution to the Ricci flow can be demonstrated by the requirement that the Finsler manifolds be isotropic for the solution to be unique.
\begin{thm}\label{isotropic}
Let $M$ be a compact isotropic Finsler manifold. Given any initial Finsler structure $F_{0}$, there exists a real number $T>0$ and a smooth one-parameter family of Finsler structures $F(t)$, for all $t\in[0,T)$, such that $F(t)$ is a unique solution to the Finslerian Ricci flow and $F(0)=F_{0}$.
\end{thm}
\section{Preliminaries and notations}
\subsection{Chern connection; A global approach}
Let $M$ be a real $n$-dimensional smooth manifold, and denote by $TM$ the tangent bundle of tangent vectors,  by  $\pi :TM_{0}\longrightarrow M$ the fiber bundle of non-zero tangent vectors and by $\pi^*TM\longrightarrow TM_0$ the pullback tangent bundle. Let $F$ be a Finsler structure on $TM_{0}$, and $g$ the related Finslerian metric \cite{BCS}. A \emph{Finsler manifold} is denoted here by the pair $(M,F)$. Any point of $TM_0$ is denoted by $z=(x,y)$, where $x=\pi z\in M$ and $y\in T_{x}M$.

Let us denote by $TTM_0$, the tangent bundle of $TM_0$ and by $\rho$, the canonical linear mapping $\rho:TTM_0\longrightarrow \pi^*TM,$ where, $ \rho=\pi_*$. For all $z\in TM_0$, $V_zTM$ is the set of all vertical vectors at $z$, that is, the set of vectors which are tangent to the fiber through $z$. Consider the decomposition $TTM_0=HTM\oplus VTM$, which allows to uniquely represent a vector field $\hat{X}\in {\cal X}(TM_0)$ as the sum of the horizontal and vertical parts namely, $\hat{X}=H\hat{X}+V\hat{X}$. The corresponding basis is denoted here by $\{\frac{\delta}{\delta {x^i}},\frac{\partial}{\partial y^{i}}\}$, where, $\frac{\delta}{\delta {x^i}}:=\frac{\partial}{\partial x^{i}}-N_{i}^{j}\frac{\partial}{\partial y^{j}}$, $N^{j}_{i}=\frac{1}{2}\frac{\partial G^j}{\partial y^i}$ and $G^i$ are the spray coefficients defined by $G^{i}=\frac{1}{4}g^{ih}(\frac{\partial^{2}F^{2}}{\partial y^{h}\partial x^{j}}y^{j}-\frac{\partial F^{2}}{\partial x^{h}})$. We denote the \emph{formal Christoffel symbols} by $\gamma^{i}_{jk}=\frac{1}{2}g^{ih}(\partial_{j}g_{hk}+\partial_{k}g_{jh}-\partial_{h}g_{jk})$ where, $\partial_{k}=\frac{\partial}{\partial x^k}$. The dual bases are denoted by $\{dx^{i},\delta y^{i}\}$ where, $\delta y^{i}:=dy^{i}+N_{j}^{i}dx^{j}$. Let us denote a global representation of the Chern connection by $\nabla:{\cal X}(TM_0)\times\Gamma(\pi^{*}TM)\longrightarrow\Gamma(\pi^{*}TM)$. Consider the linear mapping
$\mu:TTM_0\longrightarrow \pi^*TM,$ defined by $\mu(\hat{X})=\nabla_{\hat{X}}{\bf y}$ where, $\hat{X}\in TTM_0$ and ${\bf y}=y^i\frac{\partial}{\partial x^i}$ is the canonical section of $\pi^*TM$.

The connection 1-forms of Chern connection in these bases are given by $\omega^{i}_{j}=\Gamma^{i}_{jk}dx^{k}$ where, $\Gamma^{i}_{jk}=\frac{1}{2}g^{ih}(\delta_{j}g_{hk}+\delta_{k}g_{jh}-\delta_{h}g_{jk})$ and $\delta_{k}=\frac{\delta}{\delta x^{k}}$. In the sequel, all the vector fields on $TM_0$ are decorated with a hat and denoted by $\hat{X}$, $\hat{Y}$ and $\hat{Z}$, and the corresponding sections of $\pi^*TM$ by $X=\rho(\hat{X})$, $Y=\rho(\hat{Y})$ and $Z=\rho(\hat{Z})$, respectively unless otherwise specified.
The torsion freeness and almost metric compatibility of the Chern connection are given by
\begin{eqnarray}
&&\nabla_{\hat{X}}Y-\nabla_{\hat{Y}}X=\rho[\hat{X},\hat{Y}],\label{tori}\\
&&(\nabla_{\hat{Z}}g)(X,Y)=2C(\mu(\hat{Z}),X,Y),\label{gcomp}
\end{eqnarray}
respectively, where $C$ is the Cartan tensor with the components $C_{ijk}=\frac{\partial g_{ij}}{\partial y^{k}}.$ In  local coordinates on $TM$ the \emph{Chern horizontal} and \emph{vertical covariant derivatives} of an arbitrary $(1,2)$ tensor field $S$ on $\pi^{*}TM$ with the components $S^{i}_{jk}(x,y)$ on $TM$ are given by
\begin{eqnarray*}
&&\nabla_{l}S^{i}_{jk}:= \delta_{l}S^{i}_{jk}-S^{i}_{s k}\Gamma^{s}_{jl}-S^{i}_{js}\Gamma^{s}_{kl}+S^{s}_{jk}\Gamma^{i}_{s l}, 
 \\
&&\dot{\nabla}_{l}S^{i}_{jk}:=\dot{\partial}_{l}S^{i}_{jk},
\end{eqnarray*}
where, $\nabla_{l}:=\nabla_{\frac{\delta}{\delta x^l}}$ and $\dot{\nabla}_{l}:=\nabla_{\frac{\partial}{\partial y^l}}$. Horizontal metric compatibility of the Chern connection in local coordinates  translates into $\nabla_{l}g_{jk}=0$, see \cite[p.\ 45]{BCS}. The local \emph{Chern $hh$-curvature} tensor is given by
\begin{equation} \label{77}
R^{\,\,i}_{j\,\,kl}=\delta_{k}\Gamma^{i}_{\,jl}-\delta_{l}\Gamma^{i}_{\,jk}+
\Gamma^{i}_{\,hk}\Gamma^{h}_{\,jl}-\Gamma^{i}
_{\,hl}\Gamma^{h}_{\,jk},
\end{equation}
see \cite[p.\ 52]{BCS}.
In the isotropic case, the contracted $hh$-curvature tensor $R_{j\,\,il}^{\,\,i}$ is symmetric in two indices $j$ and $l$. In fact, by means of the symmetry of $R_{j\,\,il}^{\,\,i}$ in Cartan connection in the isotropic case, see \cite[p.\ 152]{HAZ}, and the relation between the Cartan and Chern $hh$-curvature tensors,  one can see that in the isotropic case, $R_{j\,\,il}^{\,\,i}$ is symmetric in $j$ and $l$.
 More intuitively, it can be part of a symmetric quadratic form.

The \emph{reduced $hh$-curvature} tensor is a connection free tensor field which is also referred to as the \emph{Riemann curvature} by certain authors. In local coordinates on $TM$, the components of the reduced $hh$-curvature tensor are given by $R^{i}_{\,\,k}:=\frac{1}{F^2}y^{j}R^{\,\,i}_{j\,\,km}y^{m}$, which are entirely expressed in terms of $x$ and $y$ derivatives of spray coefficients $G^{i}$ as follows
\begin{equation} \label{18}
R^{i}_{\,\,k}:=\frac{1}{F^2}(2\frac{\partial G^{i}}{\partial x^{k}}-\frac{\partial^{2}G^{i}}{\partial x^{j}\partial y^{k}}y^{j}+2G^{j}\frac{\partial^{2}G^{i}}{\partial y^{j}\partial y^{k}}-\frac{\partial G^{i}}{\partial y^{j}}\frac{\partial G^{j}}{\partial y^{k}}),
\end{equation}
see \cite[p.\ 66]{BCS}.

 Let $(x,y)$ be an element of $TM$ and $P(y,X)\subset T_{x}(M)$  a 2-plane
  generated by the vectors $y$ and $X$ in
$T_{x}(M) $. Then the \emph{flag curvature} $K(x,y,X)$ with respect
to  the plane $P(y,X)$ at a point $x\in M$ is defined by
$$K(x,y,X):=\frac{g(R(X,y)y,X)}{g(X,X)g(y,y)-g(X,y)^{2}},$$
where $R(X,y)y$ is the $hh$-curvature tensor. If
$K$ is independent of $X$, then the Finsler manifold $(M,F)$ is called \emph{isotropic} or \emph{space of scalar curvature}. If $K$ has no dependence on $x$ or $y$,
then the Finsler manifold is said to be of \emph{constant curvature}.

\subsection{Lie derivatives of Finsler metrics}
The Lie derivative of an arbitrary Finslerian $(0,2)$ tensor field ${\cal T}={\cal T}_{jk}(x,y)dx^{j}\otimes dx^{k}$ on $\otimes^{2}\pi^{*}TM$ with respect to an arbitrary vector field $\hat V$ on $TM_0$ is given by
\beq
(\mathcal{L}_{\hat{V}}{\cal T})(X,Y)=\hat{V}({\cal T}(X,Y))-{\cal T}(\rho[\hat V,\hat X],Y)-{\cal T}(X,\rho[\hat V,\hat Y]),\nonumber
\eeq
where, $\rho(\hat{X})=X$, $\rho(\hat{Y})=Y$ and $\hat{X},\hat{Y}\in T_{z}TM_{0}$, see \cite{JB}. The Lie derivative of a Finsler metric $g$ in direction of an arbitrary vector field $\hat V$ on $TM_0$ is given by
\beq
(\mathcal{L}_{\hat{V}}{g})(X,Y)=\hat{V}({g}(X,Y))-{g}(\rho[\hat V,\hat X],Y)-{g}(X,\rho[\hat V,\hat Y]).\nonumber
\eeq
Using the torsion freeness of Chern connection defined by (\ref{tori}), the Lie derivative of the Finsler metric $g$ becomes
\begin{eqnarray}
(\mathcal{L}_{\hat{V}}{g})(X,Y)&=&\hat{V}(g(X,Y))-g(\nabla_{\hat{V}}X-\nabla_{\hat{X}}V,Y)-g(X,\nabla_{\hat{V}}Y-\nabla_{\hat{Y}}V)\nonumber\\
&=&\hat{V}(g(X,Y))-g(\nabla_{\hat{V}}X,Y)+g(\nabla_{\hat{X}}V,Y)\nonumber\\
&&-g(X,\nabla_{\hat{V}}Y)+g(X,\nabla_{\hat{Y}}V).\label{Liederiv}
\end{eqnarray}
By the almost $g$-compatibility of Chern connection defined by (\ref{gcomp}), we have
\begin{equation*}
2C(\mu(\hat{V}),X,Y)=(\nabla_{\hat{V}}g)(X,Y)=\hat{V}(g(X,Y))-g(\nabla_{\hat{V}}X,Y)-g(X,\nabla_{\hat{V}}Y).
\end{equation*}
Therefore,
\begin{equation}\label{campatibility}
\hat{V}(g(X,Y))=2C(\mu(\hat{V}),X,Y)+g(\nabla_{\hat{V}}X,Y)+g(X,\nabla_{\hat{V}}Y).
\end{equation}
Plugging the equation (\ref{campatibility}) in (\ref{Liederiv}) we obtain
\begin{equation}\label{finalliederiv}
(\mathcal{L}_{\hat{V}}{g})(X,Y)=2C(\mu(\hat{V}),X,Y)+g(\nabla_{\hat{X}}V,Y)+g(X,\nabla_{\hat{Y}}V).
\end{equation}
Replacing $X$ and $Y$ by the canonical section ${\bf y}=y^{i}\frac{\partial}{\partial x^{i}}$ in (\ref{finalliederiv}) we obtain
\begin{equation*}
(\mathcal{L}_{\hat{V}}{g})({\bf y},{\bf y})=2C(\mu(\hat{V}),{\bf y},{\bf y})+g(\nabla_{\hat{{\bf y}}}V,{\bf y})+g({\bf y},\nabla_{\hat{{\bf y}}}V),
\end{equation*}
where, $\hat{{\bf y}}=y^{i}\frac{\delta}{\delta x^{i}}$. Using $C(\mu(\hat{V}),{\bf y},{\bf y})=0$, see \cite[p.\ 23]{BCS}, and the symmetric property of $g(\nabla_{\hat{{\bf y}}}V,{\bf y})$ one arrives at
\begin{equation}\label{global}
(\mathcal{L}_{\hat{V}}{g})({\bf y},{\bf y})=2g({\bf y},\nabla_{\hat{{\bf y}}}V),
\end{equation}
where, $V=v^i\frac{\partial}{\partial x^i}$ is a section of $\pi^{*}TM$. In local coordinates, the above equation translates into
\begin{equation*}
y^iy^j\mathcal{L}_{\hat{V}}g_{ij}=2y^iy^jg_{ik}\nabla_{j}v^{k}.
\end{equation*}
Using $\nabla_{j}g_{ik}=0$, we obtain
\begin{equation}\label{FINAL}
y^iy^j\mathcal{L}_{\hat{V}}g_{ij}=2y^iy^j\nabla_{j}v_{i},
\end{equation}
where, $v_{i}=g_{ik}v^{k}$.
\subsection{A geometric setup on $SM$}
Let $(M,F)$ be a Finsler manifold and $SM$ the quotient of $TM_0$ under the following equivalence relation: $(x,y)\sim(x,\tilde{y})$ if and only if $y, \tilde{y}$ are positive multiples of each other. In other words, $SM$ is the bundle of all directions or rays, and is called the (projective) \emph{sphere bundle}. The local coordinates $x^{1},...,x^{n}$ on $M$ induce global coordinates $y^1,...,y^n$ on each fiber $T_{x}M$, through the expansion $y=y^{i}\frac{\partial}{\partial x^i}$. Therefore $(x^i;y^i)$ is a coordinate system on $SM$, where the coordinates $y^i$ are regarded as homogeneous coordinates in the projective space. Using the canonical projection $p:SM\longrightarrow M$, one can pull the tangent bundle $TM$ back to $p^{*}TM$ which is a vector bundle with the fiber dimension $n$ over the $(2n-1)$-manifold $SM$. The vector bundle $p^{*}TM$ has a global section $l:=\frac{y^i}{F(y)}\frac{\partial}{\partial x^i}$ and a natural Riemannian metric which we here denote by $g:=g_{ij}(x,y)dx^{i}\otimes dx^{j}$.

Let $\{e_{a}=u_{a}^{i}\frac{\partial}{\partial x^i}\}$ be a $g$-orthonormal frame for $p^{*}TM$ with $e_{n}:=l$ and $\{\omega^a=v^{a}_{i}dx^i\}$ be its coframe; thus $\omega^{a}(e_{b})=\delta^{a}_{b}$. It is clear that $\omega^{n}=\frac{\partial F}{\partial y^i}dx^{i}$ which is a global section of $p^{*}T^{*}M$. Also we have $\frac{\partial}{\partial x^{i}}=v^{a}_{i}e_{a}$ and $dx^{i}=u^{i}_{a}\omega^{a}$. A basic relation between $(u^i_a)$ and $(v^a_i)$ is given by $v^a_iu_b^i=\delta^a_b$ and $u^i_av^a_j=\delta^i_j$.
For convenience, we shall also regard the $e_{a}$'s and $\omega^a$'s as local vector fields and 1-forms, respectively on $SM$. All $p^{*}TM$ related indices are raised and lowered with the metric $g$, see \cite{BAOS}. Let
\begin{eqnarray*}
&\hat{e}_{a}=u^{i}_{a}\frac{\delta}{\delta x^{i}},\quad
\hat{e}_{n+a}=u^{i}_{a}F\frac{\partial}{\partial y^i}.\\
&\omega^{a}=v^{a}_{i}dx^{i},\quad
\omega^{n+a}=v^{a}_{i}\frac{\delta y^i}{F}.
\end{eqnarray*}
Also, recall that the Latin indices run over the range $1,...,n$ and the Greek indices run from $1$ to $n-1$.
It can be shown that $\{\hat{e}_{a},\hat{e}_{n+\alpha}\}$ is a local basis for the tangent bundle $TSM$ and $\{\omega^{a},\omega^{n+\alpha}\}$ is a local basis for the cotangent bundle $T^{*}SM$. Tangent vectors on $SM$ which are annihilated by all $\{\omega^{n+\alpha}\}$'s form the horizontal sub-bundle $HSM$ of $TSM$. The fibers of $HSM$ are $n$-dimensional. On the other hand, let $VSM:=\cup_{x\in M}T(S_{x}M)$ be the vertical sub-bundle of $TSM$; its fibers are $n-1$ dimensional. The decomposition $TSM=HSM\oplus VSM$ holds because $HSM$ and $VSM$ are direct summands. The sphere bundle $SM\subset TM$ is a $(2n-1)$-dimensional Riemannian manifold equipped with the induced Riemannian metric
\begin{eqnarray*}
\hat{g}:=
\delta_{ab}\omega^{a}\otimes \omega^{b}+\delta_{\alpha\beta}\omega^{n+\alpha}\otimes\omega^
{n+\beta}.
\end{eqnarray*}
In particular, $HSM$ and $VSM$ are orthogonal with respect to $\hat{g}$.
\subsection{Integrability condition Lemma}
Here, we first recall that not all Riemannian metrics on the fibers of the pulled-back bundle come from a Finsler structure.
Eventually, any arbitrary symmetric positive-definite $(0,2)$-tensor field $g_{ij}(x,y)$ does not arise from a Finsler structure $F(x,y)$.
Intuitively, to ensure that $g_{ij}(x,y)$ are components of a Finsler structure,
the essential integrability criterion is the total symmetry of $(g_{ij})_{y^k}:=\frac{\p g_{ij}}{\p y^k}$
on all the three indices $i, j$, and $k$.
 In fact, $g_{ij}(x,y)$ follows from a Finsler structure $F(x,y)$ if and only if $(g_{ij})_{y^k}$ is totally symmetric in its three indices, see \cite[p. 56]{Bao}. Symmetry of ${({g_{ij}})_{{y^k}}}$ on all the three indices $i, j,$ and $ k$ is known in the literature as \emph{integrability condition}. Moreover, we must ensure the integrability criterion is satisfied in every step along with the Ricci flow.
To this end we consider a general evolution equation given by
\begin{equation}\label{AR}
\frac{\partial}{\partial t}g(t)=\omega(t),\quad g(0):=g_{0},
\end{equation}
where, $\omega(t):=\omega(t, x, y)$ is a family of symmetric $(0,2)$-tensors on $\pi^{*}TM$, zero-homogenous with respect to $y$ and $(\omega_{ij})_{y_k}=\frac{\partial\omega_{ij}}{\partial y^k}$ be totally symmetric in the three indices $i,j$ and $k$. The following Lemma establishes the integrability condition, see also \cite[p. 749]{YB2}.
\begin{lemma}\label{RE2}
Let $g(t)$ be a solution  to the evolution equation (\ref{AR}).
 There is a family of Finsler structures $F(t)$ on $TM$ such that,
\begin{equation}\label{Eq;IntCond}
g_{ij}(t)=\frac{1}{2}\frac{\partial^2 F(t)}{\partial y^i\partial y^j}.
\end{equation}
\end{lemma}
 \bpf
Let $M$ be a compact differential manifold, $F(t)$ a family of smooth 1-parameter Finsler structures on $TM_0$ and $g(t)$ the Hessian matrix of $F(t)$ which defines a scalar product on $\pi^{*}TM$ for every $t$.
Let $g(t)$ be a solution to the evolution equation (\ref{AR}). We have
\begin{equation}\label{ARR}
g(t)=g(0)+\int_{0}^{t}\omega(\tau)d\tau, \qquad \forall\tau\in[0,t).
\end{equation}
We show that the metric $g(t)$ satisfies the integrability condition, or equivalently there is a Finsler structure $F(t)$ on $TM_0$ satisfying \eqref{Eq;IntCond}.
For this purpose, we multiply  $g_{ij}$ by $y^i$ and $y^j$ in (\ref{ARR}),
\begin{equation*}
y^iy^jg_{ij}(t)=y^iy^jg_{ij}(0)+\int_{0}^{t}y^iy^j\omega_{ij}(\tau)d\tau.
\end{equation*}
By means of the initial condition $y^iy^jg_{ij}(0)= F^2(0)$, we get
\begin{equation}\label{EQ1}
y^iy^jg_{ij}(t)=F^2(0)+\int_{0}^{t}y^iy^j\omega_{ij}(\tau)d\tau.
\end{equation}
By positive definiteness assumption of $g_{ij}$, we put $F = (y^iy^jg_{ij})^{\frac{1}{2}}$. Twice vertical derivatives of (\ref{EQ1}) yields
\begin{equation}\label{EQ2}
\frac{1}{2}\frac{\partial^2 F^2}{\partial y^k\partial y^l}=g_{kl}(0)+\frac{1}{2}\int_{0}^{t}\frac{\partial^2}{\partial y^k\partial y^l}(y^iy^j\omega_{ij}(\tau))d\tau.
\end{equation}
On the other hand, by straightforward calculation we have
\begin{equation}\label{EQ3}
\frac{1}{2}\frac{\partial^2}{\partial y^k\partial y^l}(y^iy^j\omega_{ij}(\tau))=\frac{1}{2}\frac{\partial^2\omega_{ij}(\tau)}{\partial y^k\partial y^l}y^iy^j+\Big(\frac{\partial\omega_{ik}(\tau)}{\partial y^l}-\frac{\partial\omega_{il}(\tau)}{\partial y^k}\Big)y^i+\omega_{kl}(\tau),
\end{equation}
for all $\tau\in[0,t)$. Since $(\omega_{ij})_{y_k}$ is totally symmetric in three indices $i,j$ and $k$,  we obtain
\begin{equation*}
\frac{1}{2}\frac{\partial^2\omega_{ij}(\tau)}{\partial y^k\partial y^l}y^iy^j=0,\quad y^{i}\frac{\partial\omega_{ik}(\tau)}{\partial y^l}=0,\quad y^{i}\frac{\partial\omega_{il}(\tau)}{\partial y^k}=0.
\end{equation*}
Therefore, (\ref{EQ3}) is reduced to
\begin{equation}\label{EQ4}
\frac{1}{2}\frac{\partial^2}{\partial y^k\partial y^l}(y^iy^j\omega_{ij}(\tau))=\omega_{kl}(\tau),
\end{equation}
for all $\tau\in[0,t)$. Finally, replacing (\ref{EQ4}) in (\ref{EQ2}) we get
\begin{equation*}
\frac{1}{2}\frac{\partial^2 F^2}{\partial y^k\partial y^l}=g_{kl}(0)+\int_{o}^{t}\omega_{kl}(\tau)d\tau=g_{kl}.
\end{equation*}
Therefore, every $g_{ij}(t)$ on the fibers of the pulled-back bundle, arises from a Finsler structure, which completes the proof.
\epf
\begin{rem}
Let $g(t)$ be a solution to the evolution equation (\ref{AR}). The verification of the integrability of $g_{ij}(t)$ can also be done more simply.  To verify that $(g_{ij})_{y_k}=(g_{ik})_{y_j}$, consider  $(g_{ij})_{y_k}-(g_{ik})_{y_j}$, take the partial derivative $\partial_t$ of the both terms and insert it in the evolution equation (\ref{AR}). Therefore, from the total symmetry of $(\omega_{ij})_{y_k}$ in $i,j$ and $k$  we have
\begin{equation*}
\partial_t\Big((g_{ij})_{y_k}-(g_{ik})_{y_j}\Big)=
(w_{ij})_{y_k}-(w_{ik})_{y_j}=0.
\end{equation*}
Thus $(g_{ij})_{y_k}-(g_{ik})_{y_j}$ does not depend on $t$. At $t=0$, $g_{ij}(0)$ is the fundamental tensor of the initial data $F_0$, so $(g_{ij})_{y_k}-(g_{ik})_{y_j}=0$ at $t=0$.  The same is true for all $ t$.
\end{rem}
\section{Semi-linear parabolic equations on $SM$}
Recall that a \emph{quasi-linear} system is a system of partial differential equations where, the derivatives of the principal  order terms occur only linearly and  coefficients can depend on the derivatives of the lower order terms. It is called \emph{semi-linear} if it is quasi-linear and coefficients of the principal order terms depend only on the independent variables, but not on the solution, see \cite[p.\ 45]{Rog}. Let $M$ be an $n$-dimensional smooth manifold and $u:M\longrightarrow \mathbb{R}$ a smooth function on $M$. A \emph{semi-linear strictly parabolic} equation is a PDE of the form
\begin{eqnarray*}
\frac{\partial u}{\partial t}=a^{ij}(x,t)\frac{\partial^2 u}{\partial x^i\partial x^j}+h(x,t,u,\frac{\partial u}{\partial x^i}),
\end{eqnarray*}
where, $a^{ij}$ and $h$ are smooth functions on $M$ and for some constant $\lambda>0$ we have the parabolic assumption
\begin{eqnarray*}
a^{ij}\xi_{i}\xi_{j}\geq \lambda\parallel \xi\parallel^{2},\quad 0\neq\xi\in\chi(M),
\end{eqnarray*}
that is, all  eigenvalues of $A=(a^{ij})_{2\times2}$ have positive signs or equivalently,  $A$ is positive definite.
\begin{defn}\label{semipar}
Let $M$ be an $n$-dimensional smooth manifold and $\phi:SM\longrightarrow \mathbb{R}$ a smooth function on the sphere bundle $SM$. Consider the following  strictly parabolic semi-linear equation on $SM$;
\begin{equation*}
\frac{\partial \phi}{\partial t}=G^{_{AB}}(x,y,t)\hat{e}_{_A}\hat{e}_{_B}\phi+h(x,y,t,\rho,\hat{e}_{_A}\phi),\qquad A,B=1,2,...,(2n-1),
\end{equation*}
where, $\hat{e}_{_A}$ is a local frame for the tangent bundle $TSM$ and stand for partial derivatives on $SM$. Here, $G^{_{AB}}$ and $h$ are smooth functions on $SM$ and $G=(G^{_{AB}})$ is positive definite.
\end{defn}
More precisely, a strictly parabolic semi-linear system on $SM$ can be written in the following form.
\begin{align}\label{po}
\frac{\partial \phi}{\partial t}=&p^{ab}(x,y,t)\hat{e}_{a}\hat{e}_{b}\phi +q^{\alpha\beta}(x,y,t)\hat{e}_{n+\alpha}
\hat{e}_{n+\beta}\phi+m^{a\alpha}(x,y,t)\hat{e}_{a}\hat{e}_{n+\beta}\phi\nonumber\\
&+\textrm{lower order terms},
\end{align}
where the Latin indices $a,b,...$ and the Greek indices $\alpha,\beta,...$ run over the range $1,...,n$ and $1,...,(n-1)$, respectively and the matrix
\begin{displaymath}
G=\left(\begin{array}{c|c}
P\,\,\, & \frac{1}{2}M \\
\hline\,\,
\frac{1}{2}M^{t}\,\,\ & Q
\end{array}\right)_{(2n-1)\times(2n-1)},
\end{displaymath}
is positive definite where, $P=(p^{ab})_{n\times n},Q=(q^{\alpha\beta})_{(n-1)\times (n-1)}, M=(m^{a\alpha})_{n\times(n-1)}$.
\begin{lem}\label{mm}
Let $(M,F)$ be a Finsler manifold and $\phi:TM\longrightarrow \mathbb{R}$ a smooth function with zero homogeneity on the tangent bundle $TM$. The semi-linear differential equation
\begin{equation}\label{lili}
\frac{\partial \phi}{\partial t}=g^{ij}\frac{\delta^2\phi}{\delta x^i\delta x^j}+F^2g^{ij}\frac{\partial^2 \phi}{\partial y^i\partial y^j}+\textrm{lower order terms},
\end{equation}
is a strictly parabolic equation on $SM$.
\end{lem}
\begin{proof}
Let us denote again by $\phi$ the restriction of $\phi$ on $SM$. By definition, we have
\begin{eqnarray*}
&&\hat{e}_{a}\phi=u^i_a\frac{\delta \phi}{\delta x^i},\\
&&\hat{e}_{b}\hat{e}_{a}\phi=u^{i}_{a}u^{j}_{b}\frac{\delta^2 \phi}{\delta x^i\delta x^j}+\hat{e}_b(u^{i}_{a})(\frac{\delta \phi}{\delta x^i}).
\end{eqnarray*}
Multiplying both sides of the above equation by $g^{ab}$ we get
\begin{align*}
g^{ab}\hat{e}_{b}\hat{e}_{a}\phi &=g^{ab}u^{i}_{a}u^{j}_{b}\frac
{\delta^2 \phi}{\delta x^i\delta x^j}+g^{ab}\hat{e}_b(u^{i}_{a})
(\frac{\delta \phi}{\delta x^i})\\
&=g^{ij}\frac{\delta^2 \phi}{\delta x^i\delta x^j}+g^{ab}\hat{e}_b(u^{i}_{a})(\frac{\delta \phi}{\delta x^i}),
\end{align*}
where, $g_{ab}=g_{ij}u^i_au^j_b$. Denoting $B^c:=v^c_ig^{ab}\hat{e}_{b}(u^i_a)$ we can rewrite the expression
 $g^{ij}\frac{\delta^2 \phi}{\delta x^i\delta x^j}$ on $SM$ with respect to $\hat{e}_{a}$'s as follows
\begin{equation}\label{First}
g^{ab}\hat{e}_{b}\hat{e}_{a}\phi-B^c\hat{e}_c\phi=g^{ij}\frac{\delta^2 \phi}{\delta x^i\delta x^j}.
\end{equation}
By definition,  we also have
\begin{align*}
\hat{e}_{n+\alpha}\phi&=Fu^{i}_{\alpha}\frac{\partial \phi}{\partial y^i},\\
\hat{e}_{n+\beta}\hat{e}_{n+\alpha}\phi&=\hat{e}_{n+\beta}(Fu^{i}_{\alpha}\frac{\partial \phi}{\partial y^i})
\nn\\&=F^2u^j_\beta u^i_\alpha\frac{\partial^2 \phi}{\partial y^j\partial y^i}+F(\hat{e}_{n+\beta} u^i_{\alpha})(\frac{\partial \phi}{\partial y^i})+Fu^{j}_{\beta}(\frac{\partial F}{\partial y^j})u^i_\alpha\frac{\partial \phi}{\partial y^i}.
\end{align*}
Using the fact $u^{j}_{\beta}\frac{\partial F}{\partial y^j}=0$, see \cite[p.\ 161]{HAZ}, we have
\begin{eqnarray*}
\hat{e}_{n+\beta}\hat{e}_{n+\alpha}\phi=F^2u^j_\beta u^i_\alpha\frac{\partial^2
\phi}{\partial y^j\partial y^i}+F(\hat{e}_{n+\beta} u^i_{\alpha})(\frac{\partial \phi}{\partial y^i}).
\end{eqnarray*}
Multiplying both sides by $g^{\alpha\beta}$ and taking into account $g^{\alpha\beta}u^{i}_{\alpha}u^{j}_{\beta}=g^{ij}-y^iy^j$ we get
\begin{eqnarray*}
g^{\alpha\beta}\hat{e}_{n+\beta}\hat{e}_{n+\alpha}\phi=F^2g^{ij}\frac{\partial^2 \phi}{\partial y^j\partial y^i}+Fg^{\alpha\beta}(\hat{e}_{n+\beta}u^i_{\alpha})\frac{\partial\phi}{\partial y^i},
\end{eqnarray*}
where, $g_{\alpha\beta}=g_{ij}u^i_\alpha u^j_\beta$. If we denote $D^\gamma:=v^\gamma_iFg^{\alpha\beta}\hat{e}_{n+\beta}u^i_{\alpha}$  we can rewrite the expression  $F^2g^{ij}\frac{\partial^2 \phi}{\partial y^j\partial y^i}$ on $SM$ with respect to $\hat{e}_{n+\alpha}$'s as follows
\begin{equation}\label{second}
g^{\alpha\beta}\hat{e}_{n+\beta}\hat{e}_{n+\alpha}\phi-D^\gamma\hat{e}_{n+\gamma}\phi=F^2g^{ij}\frac{\partial^2 \phi}{\partial y^j\partial y^i}.
\end{equation}
Thus the principal order terms $g^{ij}\frac{\delta^{2}\phi}{\delta x^{i}\delta x^{j}}$ and $F^2g^{ij}\frac{\partial^{2}\phi}{\partial y^{i}\partial y^{j}}$ convert to $g^{ab}\hat{e}_{b}\hat{e}_{a}\phi-B^c\hat{e}_c\phi$
and $g^{\alpha\beta}\hat{e}_{n+\beta}\hat{e}_{n+\alpha}\phi-D^\gamma\hat{e}_{n+\gamma}\phi$ on $SM$. On the other hand, the order of the lower order terms in (\ref{lili}) do not change after rewriting them in terms of the basis $\{\hat{e}_a,\hat{e}_{n+\alpha}\}$ on $SM$. Therefore (\ref{lili}) on $SM$ is written as
\begin{equation}\label{bibi}
\frac{\partial \phi}{\partial t}=g^{ab}\hat{e}_{b}\hat{e}_{a}\phi+g^{\alpha\beta}\hat{e}_{n+\beta}\hat{e}_{n+\alpha}\phi-B^c\hat{e}_c\phi-D^\gamma\hat{e}_{n+\gamma}\phi+\textrm{lower order terms}.
\end{equation}
Using the fact that $g$ is positive definite, the coefficients
\begin{displaymath}
G=\left(\begin{array}{c|c}
g^{ab} & 0 \\
\hline
0 & g^{\alpha\beta}
\end{array}\right)_{(2n-1)\times(2n-1)},
\end{displaymath}
of principal order terms of (\ref{bibi}) is positive definite on $SM$. Therefore, by virtue of (\ref{po}) the  differential equation (\ref{bibi}) is a semi-linear strictly parabolic equation on $SM$.
\end{proof}
\section{Setting up an operator and a vector field on $SM$}
Let $(M,F)$ and $(N,\bar{F})$ be the two $n$-dimensional Finsler manifolds with the corresponding metric tensors $g$ and $h$, respectively. We decorate with a bar the geometric objects on $N$. Let $(x^i,y^i)$ and $(\bar{x}^i,\bar{y}^i)$ be the local coordinate systems on $TM$ and $TN$, respectively. Considering that $c$ is a geodesic on $(M,F)$, the natural lift of $c$ on $TM$ is given by
\begin{equation*}
\tilde{c}:t\in I\longrightarrow\tilde{c}(t)=(x^{i}(t),\frac{dx^i}{dt}(t))\in TM.
\end{equation*}
It is well known that $\tilde{c}$ is a \emph{horizontal curve} on $TM$, that is to say its tangent vector field is a horizontal vector field with the components  $\dot{\tilde{c}}(t)=\frac{dx^i}{dt}\frac{\delta}{\delta x^i}\in HTM$.
Let us consider a diffeomorphism $\varphi$ between the tangent bundles $TM$ and $TN$, such that $\bar{c}(t):=(\varphi\circ\tilde{c})(t)$ is a horizontal curve on $TN$.
\begin{align}\label{ez}
\varphi &:TM\longrightarrow TN,\\
&(x^i,y^i)\mapsto\varphi(x^i,y^i)=(\varphi^{\alpha}(x^i,y^i))\nonumber,
\end{align}
where, the components $\varphi^{\alpha}$ of $\varphi$ are $(\varphi^{\alpha}(x^i,y^i))=:(\varphi^j(x^i,y^i),\varphi^{n+j}(x^i,y^i))$, and the Latin indices $i,j, ...$ run over the range $1,\cdots,n$ and the Greek indice $\alpha$ run over the range $1,\cdots,2n$. Throughout this section, $\varphi$ takes a horizontal curve on $TM$ into a horizontal curve on $TN$.

Denote by $\Gamma^{i}_{\,jk}$ and $\bar{\Gamma}^{i}_{\,jk}$ the coefficients of the Chern connection on $(M,F)$ and $(N,\bar{F})$, respectively. Clearly
\begin{align}\label{12}
\bar{\nabla}_{\dot{\bar{c}}}\dot{\bar{c}}&
=\bar{\nabla}_{\dot{\bar{c}}}
\frac{d\bar{x}^j}{dt}\frac{\delta}{\delta \bar{x}^j}=\frac{d^{2}\bar{x}^{j}}{dt^{2}}\frac{\delta}{\delta \bar{x}^j}+\frac{d\bar{x}^j}{dt}\bar{\nabla}_{\dot{\bar{c}}}\frac{\delta}{\delta \bar{x}^j}
=(\frac{d^{2}\bar{x}^{i}}{dt^{2}}+\frac{d\bar{x}^j}{dt}\frac{d\bar{x}^k}{dt}\bar{\Gamma}^{i}_{\,jk})\frac{\delta}{\delta \bar{x}^{i}},
\end{align}
where $\dot{\bar{c}}(t):=\frac{d\bar{x}^j}{dt}\frac{\delta}{\delta \bar{x}^j} $ is the tangent vector to the horizontal curve $\bar{c}(t)$ on $TN$.
On the other hand
\begin{align*}
\frac{d\bar{x}^i}{dt}=\frac{\delta\varphi^i}{\delta x^p}\frac{dx^p}{dt}, \quad \frac{d^{2}\bar{x}^i}{dt^2}=\frac{\delta^2\varphi^i}{\delta x^p\delta x^q}\frac{dx^p}{dt}\frac{dx^q}{dt}+\frac{\delta\varphi^i}{\delta x^p}\frac{d^2x^p}{dt^2}.
\end{align*}
Replacing the last equations in (\ref{12}) leads to
\begin{equation}\label{12+1+1}
\bar{\nabla}_{\dot{\bar{c}}}\dot{\bar{c}}=(\frac{\delta^2\varphi^i}{\delta x^p\delta x^q}\frac{dx^p}{dt}\frac{dx^q}{dt}+\frac{\delta\varphi^i}{\delta x^h}\frac{d^2x^h}{dt^2}+\bar{\Gamma}^{i}_{\,jk}\frac{\delta\varphi^j}{\delta x^p}\frac{\delta\varphi^k}{\delta x^q}\frac{dx^p}{dt}\frac{dx^q}{dt})\frac{\delta}{\delta \bar{x}^{i}},
\end{equation}
where, all the indices $i,j,p,q,\cdots,$ run over the range $1,\cdots,n$. The geodesic $c$ on $(M,F)$, satisfies
\begin{equation}\label{12+1+1+1}
\frac{d^2x^h}{dt^2}+\Gamma^{h}_{\,pq}\frac{dx^p}{dt}\frac{dx^q}{dt}=0.
\end{equation}
Substituting $ \frac{d^2x^h}{dt^2}$ from the last equation in (\ref{12+1+1}), leads to
\begin{equation*}
\bar{\nabla}_{\dot{\bar{c}}}\dot{\bar{c}}=\frac{dx^p}{dt}\frac{dx^q}{dt}(\frac{\delta^2\varphi^i}{\delta x^p\delta x^q}-\frac{\delta\varphi^i}{\delta x^h}\Gamma^{h}_{\,pq}+\bar{\Gamma}^{i}_{\,jk}\frac{\delta\varphi^j}{\delta x^p}\frac{\delta\varphi^k}{\delta x^q})\frac{\delta}{\delta \bar{x}^{i}}.
\end{equation*}
Next, let
\begin{eqnarray*}
A^{i}_{pq}:=\frac{\delta^2\varphi^i}{\delta x^p\delta x^q}-\frac{\delta\varphi^i}{\delta x^h}\Gamma^{h}_{\,pq}+\bar{\Gamma}^{i}_{\,jk}\frac{\delta\varphi^j}{\delta x^p}\frac{\delta\varphi^k}{\delta x^q}+F^2\frac{\partial^2\varphi^i}{\partial y^p\partial y^q}.
\end{eqnarray*}
Contracting $A^i_{pq}$ with $g^{pq}$ leads to the following operators
\begin{equation} \label{15}
(\Phi_{g,h}\varphi)^{i}:=g^{pq}(\frac{\delta^2\varphi^i}{\delta x^p\delta x^q}+F^2\frac{\partial^2\varphi^i}{\partial y^p\partial y^q}-\frac{\delta\varphi^i}{\delta x^h}\Gamma^{h}_{\,pq}+\bar{\Gamma}^{i}_{\,jk}\frac{\delta\varphi^j}{\delta x^p}\frac{\delta\varphi^k}{\delta x^q}),
\end{equation}
where, $(\Phi_{g,h}\varphi)^{i}=g^{pq}A^{i}_{pq}$ and $i=1,...,n$. For the indices greater than ``$n$" we consider the following operator.
\begin{equation} \label{15+1+1}
(\Phi_{g,h}\varphi)^{n+i}:=g^{pq}(\frac{\delta^2\varphi^{n+i}}{\delta x^p\delta x^q}+F^2\frac{\partial^2\varphi^{n+i}}{\partial y^p\partial y^q}+\frac{\partial\varphi^{n+i}}{\partial y^k}\frac{\delta N^k_q}{\delta x^p}),
\end{equation}
where, $i=1,\cdots,n$. Summarizing the above definitions we have
\begin{equation*}
(\Phi_{g,h}\varphi)^{\alpha}=\left\{
\begin{array}{l}
(\Phi_{g,h}\varphi)^{i}\qquad \alpha=i\cr
(\Phi_{g,h}\varphi)^{n+i}\qquad \alpha=n+i,
\end{array}
\right.
\end{equation*}
where, $i=1,...,n$. Next, we show the operator $(\Phi_{g,h}\varphi)^{\alpha}$ is invariant under all diffeomorphisms on $TM$.
Our approach to defining the above operator on $SM$ and the following key lemma is an extension of that presented in \cite{BRE}.
\begin{lem} \label{main2}
Let $(M,F)$ and $(N,\bar{F})$ be two $n$-dimensional Finsler manifolds with the corresponding metric tensors $g$ and $h$, respectively. If $\psi$ is a diffeomorphism from $TM$ to itself, then it leaves invariant the operator
$(\Phi_{g,h}\varphi)^\alpha$, in the sense that
\begin{eqnarray*}
(\Phi_{\psi^{*}(g),h}\psi^{*}\varphi)^\alpha\mid_{(\tilde{x},\tilde{y})}=(\Phi_{g,h}\varphi)^\alpha\mid_{(x,y)},
\end{eqnarray*}
where, $\tilde{x}^i=\psi^{*}x^i$, $\tilde{y}^i=\psi^{*}y^i$ and  $\alpha=1,\cdots,2n$.
\end{lem}
\begin{proof}
Let $(x^i,y^i)$ and $(\bar{x}^i,\bar{y}^i)$ be the two local coordinate systems on $TM$ and $TN$, respectively. Letting $\tilde{x}^i=\psi^{*}x^i$, $\tilde{y}^i=\psi^{*}y^i$, and considering the diffeomorphism $\varphi$ given in \eqref{ez} for $\alpha=i$, we have
\begin{align*}
(\Phi_{g,h}\varphi)^i\mid_{(x,y)}&=g^{pq}(x,y)\Big (\frac{\delta^{2}\varphi^{i}}{\delta x^{p} \delta x^{q}}(x,y)+F^2(x,y)\frac{\partial^2\varphi^i}{\partial y^p\partial y^q}(x,y)
-\frac{\delta \varphi^{i}}{\delta x^{k}}(x,y)\Gamma^{k}_{pq}(x,y)\\
&+\bar{\Gamma}^{i}_{jk}(\bar{x},\bar{y})\frac{\delta \varphi^{j}}{\delta x^{p}}(x,y)\frac{\delta \varphi^{k}}{\delta x^{q}}(x,y)\Big).
\end{align*}
Using the diffeomorphism $\psi$ and the pull-back $\psi^*$ we have
\begin{align*}
(\Phi_{g,h}\varphi)^i\mid_{(x,y)}&=g^{pq}(\psi(\tilde{x},\tilde{y}))\Big(\frac{\delta^{2}\varphi^{i}}{\delta x^{p} \delta x^{q}}(\psi(\tilde{x},\tilde{y}))+F^2(\psi(\tilde{x},\tilde{y}))\frac{\partial^2\varphi^i}{\partial y^p\partial y^q}(\psi(\tilde{x},\tilde{y})) \\
&\qquad-\frac{\delta \varphi^{i}}{\delta x^{k}}(\psi(\tilde{x},\tilde{y}))\Gamma^{k}_{pq}(\psi(\tilde{x},\tilde{y}))+\bar{\Gamma}^{i}_{jk}(\bar{x},\bar{y})\frac{\delta \varphi^{j}}{\delta x^{p}}(\psi(\tilde{x},\tilde{y}))\frac{\delta \varphi^{k}}{\delta x^{q}}(\psi(\tilde{x},\tilde{y}))\Big)\\
&=(\psi^{*}
g)^{pq}(\tilde{x},\tilde{y})\Big(\frac{\delta^{2}(\psi^{*}\varphi)^{i}}{\delta \tilde{x}^{p} \delta \tilde{x}^{q}}(\tilde{x},\tilde{y})+(\psi^{*}F^2)(\tilde{x},\tilde{y})\frac{\partial^{2}(\psi^{*}\varphi)^{i}}{\partial \tilde{y}^{p} \partial \tilde{y}^{q}}(\tilde{x},\tilde{y})\\
&\qquad-\frac{\delta (\psi^{*}\varphi )^{i}}{\delta \tilde{x}^{k}}(\tilde{x},\tilde{y})\Gamma(\psi^{*}g)^{k}_{pq}(\tilde{x},\tilde{y})
+\bar{\Gamma}^{i}_{jk}(\bar{x},\bar{y})\frac{\delta (\psi^{*}\varphi)^{j}}{\delta \tilde{x}^{p}}(\tilde{x},\tilde{y})\frac{\delta (\psi^{*}\varphi)^{k}}{\delta \tilde{x}^{q}}(\tilde{x},\tilde{y})\Big) \\
&=(\Phi_{\psi^{*}(g),h}\psi^{*}\varphi)^i\mid_{(\tilde{x},\tilde{y})}.\nonumber
\end{align*}
Similarly, for $\alpha=n+i$, one can show that
\begin{equation*}
(\Phi_{g,h}\varphi)^{n+i}\mid_{(x,y)}=(\Phi_{\psi^{*}(g),h}\psi^{*}\varphi)^{n+i}\mid_{(\tilde{x},\tilde{y})},
\end{equation*}
where, $i=1,\cdots,n.$ This completes the proof.
\end{proof}
\begin{rem} \label{main3+1}
Let $(M,F)$ and $(N,\bar{F})$ be two $n$-dimensional Finsler manifolds with the corresponding metric tensors $g$ and $h$, respectively. Let $\varphi:TM\longrightarrow TN$, $\varphi(x^i,y^i)=(\varphi^\alpha(x^i,y^i))$, $\alpha=1,..,2n$, be a diffeomorphism that takes a horizontal curve to a horizontal curve. Given $\varphi_{0}:TM \longrightarrow TN$, we consider the following evolution equation
\begin{equation} \label{1666}
\frac{\partial}{\partial t}\varphi^\alpha=(\Phi_{g,h}\varphi)^\alpha,\hspace{0.6cm} \varphi_{(0)}=\varphi_{0}.
\end{equation}
By restricting $\varphi^\alpha$'s to $SM$ and using Lemma \ref{mm}, one can see that (\ref{1666}) is a strictly parabolic system. Hence, there is a unique solution for (\ref{1666}) in a short-time.
\end{rem}
\begin{cor} \label{main3}
Let $(M,\tilde{F})$ and $(N,\bar{F})$ be two $n$-dimensional Finsler manifolds with corresponding metric tensors $\tilde{g}$ and $h$, respectively. Let $N=M$ and $\varphi$ be the identity map $\varphi=Id:TM\longrightarrow TM$, $\varphi(x^i,y^i)=(x^i,y^i)$, then we have
\begin{equation*} 
(\Phi_{\tilde{g},h}Id)^{\alpha}=(\Phi_{\tilde{g},h}Id)^{i}=\tilde{g}^{pq}(-\tilde{\Gamma}^{i}_{pq}+\bar{\Gamma}^{i}_{pq}),\quad \alpha=i,
\end{equation*}
where, $i=1,\cdots,n$ and $\tilde{\Gamma}^{i}_{pq}, \bar{\Gamma}^{i}_{pq}$ are the coefficients of horizontal covariant derivatives of the Chern connection with respect to the metrics $\tilde{g}$ and $h$, respectively.
\end{cor}
Let $\xi$ be a vector field on $TM$ with the components
\begin{align}\label{def;vect}
\xi:=(\Phi_{\tilde{g},h}Id)^{i}\frac{\partial}{\partial x^i}=\tilde{g}^{pq}(-\tilde{\Gamma}^{i}_{pq}+\bar{\Gamma}^{i}_{pq})\frac{\partial}{\partial x^i}.
\end{align}
Using the fact that the difference between any two connections is a tensor field, $\xi$ is a globally well-defined vector field. It is easy to verify that the components of $\xi$ are homogeneous of degree zero on $y$. Thus $\xi$ can be considered as a vector field on $SM$.
\section{Ricci-DeTurck flow on Finsler manifolds}
Here we study the Ricci-DeTurck flow and its existence of solution. More intuitively, there are several well-known definitions of the Ricci tensor in Finsler geometry. For instance, H. Akbar-Zadeh has considered two Ricci tensors on Finsler manifolds in his works, namely, one is defined by $ Ric_{ij}:=[\frac{1}{2}F^{2}\mathcal{R}ic]_{y^{i}y^{j}}$ where, $\mathcal{R}ic$ is the \emph{Ricci scalar} defined by $\mathcal{R}ic:=g^{ik}R_{ik}=R^{i}_{\,\,i}$ and $R^{i}_{\,\,k}$ are defined by (\ref{18}), see \cite[p.\ 192]{BCS}. Another Ricci tensor is defined by
$Rc_{ij}:=\frac{1}{2} (\textsf{R}_{ij}+\textsf{R}_{ji})$, where $ \textsf{R}_{ij}$ is the trace of $hh$-curvature defined by $ \textsf{R}_{ij}=R^{\,\,l}_{i\,\,lj}$. The difference between these two Ricci tensors is the additional term $\frac{1}{2}y^k\frac{\partial \textsf{R}_{jk}}{\partial y^i}$ that appeared in the first definition. More precisely, we have $Ric_{ij}-Rc_{ij}=\frac{1}{2}y^k\frac{\partial \textsf{R}_{jk}}{\partial y^i}$.

In \cite{Bao},
 D. Bao, based on the first definition of the Ricci tensor, considered the following evolutionary differential equation as a \emph{Ricci flow} equation in Finsler geometry:
\begin{equation}\label{20002}
\frac{\partial}{\partial t}g_{jk}(t)=-2Ric_{jk},\hspace{0.6cm}g_{(t=0)}=g_{0},
\end{equation}
where, $g_{jk}(t)$ are a  family of Finslerian metrics defined on $\pi^{*}TM\times[0,T)$.
Contracting (\ref{20002}) with $y^{j}y^{k}$, via Euler's theorem, leads to $\frac{\partial}{\partial t}F^{2}=-2F^{2}\mathcal{R}ic$. That is,
\begin{equation} \label{20}
\frac{\partial}{\partial t}\log F(t)=-\mathcal{R}ic,\hspace{0.6cm}F(t=0):=F_{0},
\end{equation}
where, $F_{0}$ is the initial Finsler structure, see more details on \cite{Bao}. Here and everywhere in the present work we consider the first Akbar-Zadeh's definition of Ricci tensor and the related Ricci flow (\ref{20}). One of the advantages of the Ricci quantity $Ric_{ij}$, used in the present work is its independence with respect to the choice of Cartan, Berwald or Chern connections.
\begin{defn}\label{Main4}
Let $M$ be a compact manifold with a fixed background Finsler structure $\bar{F}$ and the related Finsler metric $h$. Assume that for all $t\in [0,T)$, $\tilde{F}(t)$ is a one-parameter family of Finsler structures defined on $TM$ and $\tilde{g}(t)$ is the tensor metric related to $\tilde{F}(t)$. We say that $\tilde{F}(t)$ is a solution to the Finslerian Ricci-DeTurck flow if
\begin{equation} \label{22}
\frac{\partial}{\partial t}\tilde{F}^{2}(t)=-2\tilde{F}^{2}(t)\mathcal{R}ic(\tilde{g}(t))-\mathcal{L}_{\xi}\tilde{F}^{2}(t),
\end{equation}
where, $\mathcal{L}_{\xi}$ is the Lie derivative with respect to the vector field $\xi=(\Phi_{\tilde{g}(t),h}Id)^{i}\frac{\partial}{\partial x^i}$ on $SM$ as mentioned previously in \eqref{def;vect}.
\end{defn}
The following theorem shows that the Ricci-DeTurck flow (\ref{22}) is well defined and has a solution in a short-time interval.\\
{\bf Proof of Theorem \ref{main8}.}
Let $M$ be a compact manifold with a fixed background Finsler structure $\bar{F}$ and the related Finsler metric $h$. Here, all the indices run over the range $1,...,n$. The Ricci-DeTurck flow (\ref{22}) can be written in the following form
\begin{equation}\label{28}
y^{p}y^{q}\frac{\partial}{\partial t}\tilde{g}_{pq}(t)=-2\tilde{F}^{2}(t)\mathcal{R}ic(\tilde{g}(t))-\mathcal{L}_{\xi}y^{p}y^{q}\tilde{g}_{pq}(t),
\end{equation}
where, $\tilde{g}(t)$ is the metric tensor related to $\tilde{F}(t)$. Also we have
\begin{equation*}
\mathcal{L}_{\xi}(y^{p}y^{q}\tilde{g}_{pq})=y^{p}y^{q}\mathcal{L}
_{\xi}\tilde{g}_{pq}+2y^{p}\tilde{g}_{pq}\mathcal{L}_{\xi}y^{q}.
\end{equation*}
Therefore, (\ref{28}) becomes
\begin{equation}\label{2800}
y^{p}y^{q}\frac{\partial}{\partial t}\tilde{g}_{pq}(t)=-2\tilde{F}^{2}(t)\mathcal{R}ic(\tilde{g}(t))-y^{p}y^{q}\mathcal{L}
_{\xi}\tilde{g}_{pq}-2y^{p}\tilde{g}_{pq}\mathcal{L}_{\xi}y^{q}.
\end{equation}
By means of the Lie derivative formula (\ref{FINAL}) along $\xi$ we have
\begin{equation}\label{30}
y^py^q\mathcal{L}_{\xi}\tilde{g}_{pq}=2y^py^q\nabla_{p}\xi_{q},
\end{equation}
where, $\nabla_{p}$ is the horizontal covariant derivative in Chern connection. The $h$-metric compatibility of the Chern connection, yields
\begin{eqnarray*}
\nabla_{p}\xi_{q}=(\nabla_{p}\tilde{g}_{ql}\xi^{l})=
\tilde{g}_{ql}(\nabla_{p}\xi^{l}).
\end{eqnarray*}
As mentioned earlier, if we denote the coefficients of horizontal covariant derivatives of Chern connection with respect to the metric tensors $h$ and $\tilde{g}$ by $\Gamma(h)$ and $\Gamma(\tilde{g})$ respectively, then by the definition of $\xi$  given in \eqref{def;vect}, the above equation becomes
\begin{align*}
\nabla_{p}\xi_{q}&=\tilde{g}_{ql}(\delta_{p}\xi^{l}+\Gamma(\tilde{g})^{l}_{pw}\xi^{w})\\
&=\tilde{g}_{ql}[\delta_{p}(\tilde{g}^{mn}(\Gamma(h)^{l}_{mn}-\Gamma(\tilde{g})^{l}_{mn}))]+\tilde{g}_{ql}\Gamma(\tilde{g})^{l}_{pw}\xi^{w}\\
&=\tilde{g}_{ql}[(\delta_{p}\tilde{g}^{mn})(\Gamma(h)^{l}_{mn}-\Gamma(\tilde{g})^{l}_{mn})+\tilde{g}^{mn}\delta_{p}(\Gamma(h)^{l}_{mn})-\tilde{g}^{mn}\delta_{p}(\Gamma(\tilde{g})^{l}_{mn})]\nonumber\\
&\quad +\tilde{g}_{ql}\Gamma(\tilde{g})^{l}_{pw}\xi^{w}.
\end{align*}
According to the horizontal derivative of $\Gamma(\tilde{g})^{l}_{mn}$ in the last equation, we get
\begin{align*}
	\nabla_{p}\xi_{q}
=&\frac{1}{2}\tilde{g}^{mn}(\delta_{p}\delta_{q}\tilde{g}_{mn}
-\delta_{p}\delta_{n}\tilde{g}_{qm}-\delta_{p}\delta_{m}\tilde{g}_{qn})-\frac{1}{2}\tilde{g}_{ql}\tilde{g}^{mn}(\delta_{p}\tilde{g}^{ls})
(\delta_{n}\tilde{g}_{sm}-\delta_{s}\tilde{g}_{mn}+\delta_{m}\tilde{g}_{ns})\\
& +\tilde{g}_{ql}(\delta_{p}\tilde{g}^{mn})(\Gamma(h)^{l}_{mn}-\Gamma(\tilde{g})^{l}_{mn})
 +\tilde{g}_{ql}\tilde{g}^{mn}\delta_{p}(\Gamma(h)^{l}_{mn})+\tilde{g}_{ql}\Gamma(\tilde{g})^{l}_{pw}\xi^{w}.\\
\end{align*}

Using the last equation, (\ref{30}) is written
\begin{align} \label{31}
y^{p}y^{q}\mathcal{L}_{\xi}\tilde{g}_{pq}=&y^py^q\tilde{g}^{mn}(\delta_{p}\delta_{q}\tilde{g}_{mn}
-\delta_{p}\delta_{n}\tilde{g}_{qm}-\delta_{p}\delta_{m}\tilde{g}_{qn})\nonumber\\
&-y^py^q\tilde{g}_{ql}\tilde{g}^{mn}(\delta_{p}\tilde{g}^{ls})
(\delta_{n}\tilde{g}_{sm}-\delta_{s}\tilde{g}_{mn}+\delta_{m}\tilde{g}_{ns})\nonumber\\
&+2y^py^q\tilde{g}_{ql}(\delta_{p}\tilde{g}^{mn})(\Gamma(h)^{l}_{mn}-\Gamma(\tilde{g})^{l}_{mn})\nonumber\\
&+2y^py^q\tilde{g}_{ql}\tilde{g}^{mn}\delta_{p}(\Gamma(h)^{l}_{mn})+2y^py^q\tilde{g}_{ql}\Gamma
(\tilde{g})^{l}_{pw}\xi^{w}.
\end{align}
Also we have
\begin{equation} \label{32}
-2\tilde{F}^{2}\mathcal{R}ic(\tilde{g})=-2\tilde{F}^{2}R^{n}_{\,\,n}
=-2\tilde{F}^{2}l^{q}R^{\,\,n}_{q\,\,np}l^{p},
\end{equation}
where, $R^{\,\,n}_{q\,\,np}$ are the components of the $hh$-curvature tensor of Chern connection and $l^{q}=\frac{y^{q}}{\tilde{F}}$ are the components of the Liouville vector field. Replacing (\ref{77}) in (\ref{32}) and using the definition of $\Gamma(\tilde{g})$, yields
\begin{align*}
-2\tilde{F}^{2}\mathcal{R}ic(\tilde{g})&=-2\tilde{F}^{2}l^{q}R^{\,\,n}_{q\,\,np}l
^{p}\nonumber\\
&=-2y^{p}y^{q}(\delta_{n}\Gamma^{n}_{qp}(\tilde{g})-\delta_{p}\Gamma^{n}_{qn}(\tilde{g})
+\Gamma^{n}_{mn}(\tilde{g})\Gamma^{m}_{qp}(\tilde{g})-\Gamma^{n}_{mp}(\tilde{g})\Gamma^{m}
_{qn}(\tilde{g}))\nonumber \\
&=-2y^{p}y^{q}[\delta_{n}(\frac{1}{2}\tilde{g}^{mn}(\delta_{q}\tilde{g}_{mp}+\delta_{p}\tilde{g}_{mq}-\delta_{m}\tilde{g}_{pq}))]\nonumber\\
&\quad +2y^{p}y^{q}[\delta_{p}(\frac{1}{2}\tilde{g}^{mn}(\delta_{q}\tilde{g}_{mn}+\delta_{n}\tilde{g}_{qm}-\delta_{m}\tilde{g}_{qn}))] \nonumber\\
&\quad -2y^{p}y^{q}(\Gamma^{n}_{mn}(\tilde{g})\Gamma^{m}_{qp}(\tilde{g})-\Gamma^{n}_{mp}(\tilde{g})\Gamma^
{m}_{qn}(\tilde{g})).
\end{align*}
By applying the $\delta_{p}$ derivative we have
\begin{align} \label{34}
-2\tilde{F}^{2}\mathcal{R}ic(\tilde{g})=&y^{p}y^{q}\tilde{g}^{mn}(\delta_{p}\delta_{q}\tilde{g}_{mn}+\delta_{n}\delta_{m}
\tilde{g}_{pq}-\delta_{p}\delta_{m}\tilde{g}_{qn}-\delta_{n}\delta_{q}
\tilde{g}_{mp})\nonumber\\
&-y^{p}y^{q}(\delta_{n}\tilde{g}^{nm})(\delta_{q}\tilde{g}_{mp}+\delta_{p}\tilde{g}_{qm}
-\delta_{m}\tilde{g}_{pq})\nonumber\\
&+y^{p}y^{q}(\delta_{p}\tilde{g}^{mn})(\delta_{q}\tilde{g}_{mn}+\delta_{n}\tilde{g}_{qm}
-\delta_{m}\tilde{g}_{qn})\nonumber\\
&-2y^{p}y^{q}(\Gamma^{n}_{mn}(\tilde{g})\Gamma^{m}_{qp}(\tilde{g})-\Gamma^{n}_{mp}(\tilde{g})
\Gamma^{m}_{qn}(\tilde{g})).
\end{align}
Substituting (\ref{31}) and (\ref{34}) in (\ref{2800}), we obtain
\begin{align} \label{36}
y^{p}y^{q}\frac{\partial}{\partial t}\tilde{g}_{pq}(t)=&y^py^q\tilde{g}^{mn}\delta_{n}\delta_{m}\tilde{g}_{pq}\\
&-y^{p}y^{q}(\delta_{n}\tilde{g}^{nm})(\delta_{q}\tilde{g}_{mp}+\delta_{p}\tilde{g}_{qm}
-\delta_{m}\tilde{g}_{pq})\nonumber\\
&+y^{p}y^{q}(\delta_{p}\tilde{g}^{mn})(\delta_{q}\tilde{g}_{mn}+\delta_{n}\tilde{g}_{qm}
-\delta_{m}\tilde{g}_{qn})\nonumber\\
&-2y^{p}y^{q}(\Gamma^{n}_{mn}(\tilde{g})\Gamma^{m}_{qp}(\tilde{g})-\Gamma^{n}_{mp}(\tilde{g})
\Gamma^{m}_{qn}(\tilde{g}))\nonumber\\
&+y^py^q\tilde{g}_{ql}\tilde{g}^{mn}(\delta_{p}\tilde{g}^{ls})
(\delta_{n}\tilde{g}_{sm}-\delta_{s}\tilde{g}_{mn}+\delta_{m}\tilde{g}_{ns})\nonumber\\
&-2y^py^q\tilde{g}_{ql}(\delta_{p}\tilde{g}^{mn})(\Gamma(h)^{l}_{mn}-\Gamma(\tilde{g})^{l}_{mn})\nonumber\\
&-2y^py^q\tilde{g}_{ql}\tilde{g}^{mn}\delta_{p}(\Gamma(h)^{l}_{mn})-2y^py^q\tilde{g}_{ql}\Gamma
(\tilde{g})^{l}_{pw}\xi^{w}\nonumber\\
&-2y^{p}\tilde{g}_{pq}\mathcal{L}_{\xi}y^{q}.\nonumber
\end{align}
Using Euler's theorem yields
\begin{equation} \label{RE3}
y^{p}y^{q}\frac{\partial^{2}\tilde{g}_{pq}}{\partial y^{n}\partial y^{m}}=\frac{\partial ^{2}}{\partial y^{n}\partial y^{m}}(y^{p}y^{q}\tilde{g}_{pq})-2\tilde{g}_{nm}=0.
\end{equation}
In order to get a strictly parabolic system, by virtue of (\ref{RE3}) we add the zero term $\tilde{F}^2 y^{p}y^{q}\tilde{g}^{mn}\frac{\partial^{2}\tilde{g}_{pq}}{\partial y^{n}\partial y^{m}}=0$  to the right hand side of (\ref{36}). Therefore, we have
\begin{align} \label{36+11}
y^{p}y^{q}\Big(&\frac{\partial}{\partial t}\tilde{g}_{pq}(t)-\tilde{g}^{mn}\delta_{n}\delta_{m}\tilde{g}_{pq}-\tilde{F}^2 \tilde{g}^{mn}\frac{\partial^{2}\tilde{g}_{pq}}{\partial y^{n}\partial y^{m}}\\
&+(\delta_{n}\tilde{g}^{nm})(\delta_{q}\tilde{g}_{mp}+\delta_{p}\tilde{g}_{qm}
-\delta_{m}\tilde{g}_{pq})\nonumber\\
&-(\delta_{p}\tilde{g}^{mn})(\delta_{q}\tilde{g}_{mn}+\delta_{n}\tilde{g}_{qm}
-\delta_{m}\tilde{g}_{qn})\nonumber\\
&+2(\Gamma^{n}_{mn}(\tilde{g})\Gamma^{m}_{qp}(\tilde{g})-\Gamma^{n}_{mp}(\tilde{g})
\Gamma^{m}_{qn}(\tilde{g}))\nonumber\\
&-\tilde{g}_{ql}\tilde{g}^{mn}(\delta_{p}\tilde{g}^{ls})
(\delta_{n}\tilde{g}_{sm}-\delta_{s}\tilde{g}_{mn}+\delta_{m}\tilde{g}_{ns})\nonumber\\
&+2\tilde{g}_{ql}(\delta_{p}\tilde{g}^{mn})(\Gamma(h)^{l}_{mn}-\Gamma(\tilde{g})^{l}_{mn})\nonumber\\
&+2\tilde{g}_{ql}\tilde{g}^{mn}\delta_{p}(\Gamma(h)^{l}_{mn})-2y^py^q\tilde{g}_{ql}\Gamma
(\tilde{g})^{l}_{pw}\xi^{w}+2\frac{l_{q}}{F}\tilde{g}_{np}\mathcal{L}_{\xi}y^{n}\Big)=0.\nonumber
\end{align}
On the other hand, applying  twice the vector field  $\frac{\delta}{\delta x^{n}}$ on the components of metric tensor  $\tilde{g}_{pq}$ yields
\begin{align*} \label{RE1}
\delta_{n}\delta_{m}\tilde{g}_{pq}=&\frac{\partial^{2}\tilde{g}_{pq}}{\partial x^{n}\partial x^{m}}-\frac{\partial N^{k}_{m}}{\partial x^{n}}\frac{\partial \tilde{g}_{pq}}{\partial y^{k}}-N^{k}_{m}\frac{\partial^{2}\tilde{g}_{pq}}{\partial x^{n}\partial y^{k}}-N^{l}_{n}\frac{\partial^{2}\tilde{g}_{pq}}{\partial y^{l}\partial x^{m}}\nonumber\\
&+N^{l}_{n}\frac{\partial N_{m}^{k}}{\partial y^{l}}\frac{\tilde{g}_{pq}}{\partial y^{k}}+
N^{l}_{n}N_{m}^k\frac{\partial^2\tilde{g}_{pq}}{\partial y^k\partial y^l}.\nonumber
\end{align*}
Convecting the last equation with $y^{p}y^{q}$ and using (\ref{RE3}) we have
\begin{equation*} 
y^{p}y^{q}\delta_{n}\delta_{m}\tilde{g}_{pq}
=y^{p}y^{q}\frac{\partial^{2}\tilde{g}_{pq}}{\partial x^{n}\partial x^{m}}.
\end{equation*}
Remark that, in the term $y^py^q\tilde{g}^{mn}\delta_{n}\delta_{m}\tilde{g}_{pq}$ in (\ref{36}), there is no term containing derivatives of $\tilde{g}$ except $y^{p}y^{q}\frac{\partial^{2}\tilde{g}_{pq}}{\partial x^{n}\partial x^{m}}$. One can rewrite (\ref{36+11}) as follows
\begin{align} \label{36+111}
y^{p}y^{q}\Big(\frac{\partial}{\partial t}\tilde{g}_{pq}(t)-\tilde{g}^{mn}\delta_{n}\delta_{m}\tilde{g}_{pq}-\tilde{F}^2 \tilde{g}^{mn}\frac{\partial^{2}\tilde{g}_{pq}}{\partial y^{n}\partial y^{m}}
+\textrm{lower order terms}\Big)=0.
\end{align}
Let us  set the term between the parentheses in \eqref{36+111} equal to zero, that is,
\begin{equation}\label{36+1111}
\frac{\partial}{\partial t}\tilde{g}_{pq}(t)-\tilde{g}^{mn}\delta_{n}\delta_{m}\tilde{g}_{pq}-\tilde{F}^2 \tilde{g}^{mn}\frac{\partial^{2}\tilde{g}_{pq}}{\partial y^{n}\partial y^{m}}
+\textrm{lower order terms}=0
\end{equation}
By restricting the metric tensor $\tilde{g}$ on $p^{*}TM$ and using Lemma \ref{mm} we can rewrite (\ref{36+1111}) in terms of the basis $\{\hat{e}_{a},\hat{e}_{n+\alpha}\}$ on $SM$ as follows
\begin{eqnarray}\label{asli}
\frac{\partial}{\partial t}\tilde{g}_{pq}=\tilde{g}^{ab}\hat{e}_{b}\hat{e}_{a}\tilde{g}_{pq}+
\tilde{g}^{\alpha\beta}\hat{e}_{n+\beta}\hat{e}_{n+\alpha}\tilde{g}_{pq}
-B^c\hat{e}_c\tilde{g}_{pq}
-D^\gamma\hat{e}_{n+\gamma}\tilde{g}_{pq}\nn\\
+\textrm{lower order terms}=0,
\end{eqnarray}
where, $B^c:=v^c_i\tilde{g}^{ab}\hat{e}_{b}(u^i_a)$ and $D^\gamma:=v^\gamma_i\tilde{F}\tilde{g}^{\alpha\beta}\hat{e}_{n+\beta}u^i_{\alpha}$
as mentioned in Lemma \ref{mm}.

By hypothesis, $M$ is compact and so is the sphere bundle $SM$. It is well known that  the metric tensor $\tilde{g}^{mn}$ remains positive definite along the Ricci flow, see \cite{YB2}, Corollary 3.7. Since the coefficients of the principal (second) order terms of (\ref{asli}) are  positive definite, by Definition \ref{semipar}, it is a  strictly parabolic semi-linear system on $SM$. Therefore, the standard existence and uniqueness theorem for parabolic systems on compact domains implies that, (\ref{asli}) has a unique solution on $SM$. The equation (\ref{asli}) is a special case of the general flow equation (\ref{AR}) and $\tilde{g}(t)$ is a solution. Therefore, by means of Lemma \ref{RE2}, $\tilde{g}(t)$ satisfies the integrability condition or equivalently, there exists a Finsler structure $\tilde{F}(t)$ on $TM$ such that $\tilde{g}_{ij}=\frac{1}{2}\frac{\partial^2 \tilde{F}}{\partial y^i\partial y^j}$.
Hence, $\tilde{g}$ is a Finsler metric and determines a Finsler structure  $\tilde{F}^{2}:=\tilde{g}_{pq}y^{p}y^{p}$ which is a solution to the Finsler Ricci-DeTurck flow. This completes the proof of Theorem \ref{main8}.\hspace{\stretch{1}}$\Box$
\begin{rem}\label{remarkunique}
In Theorem \ref{main8}, as it is mentioned earlier in the preliminaries, if the Finsler manifold $M$ is  isotropic, then $R^{\,\,n}_{q\,\,np}$ is symmetric with respect to the indices $p$ and $q$.
Therefore, by means of symmetry of $\mathcal{L}_{\xi}(y^{p}y^{q}\tilde{g}_{pq})$ on the indices $p$ and $q$, we conclude that the term in the brackets in (\ref{36+111}) is symmetric on the indices $p$ and $q$. If a symmetric bilinear form vanishes on the diagonal, then by the polarization identity it vanishes identically. Hence, according to (\ref{36+111}) we have
\begin{equation}
\frac{\partial}{\partial t}\tilde{g}_{pq}(t)-\tilde{g}^{mn}\delta_{n}\delta_{m}\tilde{g}_{pq}-\tilde{F}^2 \tilde{g}^{mn}\frac{\partial^{2}\tilde{g}_{pq}}{\partial y^{n}\partial y^{m}}
+\textrm{lower order terms}=0.
\end{equation}
As shown in the proof of Theorem \ref{main8}, by restricting the metric tensor $\tilde{g}$ on $p^{*}TM$ and using Lemma \ref{mm} we can rewrite (\ref{36+1111}) on $SM$ as (\ref{asli}) which is a strictly parabolic semi-linear  system on $SM$. Therefore, the standard existence and uniqueness theorem for parabolic systems on compact domains implies that, (\ref{asli}) has a unique solution on $SM$. Hence, Finsler structure $\tilde{F}^{2}:=\tilde{g}_{pq}y^{p}y^{p}$ is a unique solution to the Finsler Ricci-DeTurck flow. Therefore, in the isotropic case, the solution to the \eqref{22} is unique.
\end{rem}
\section{Short-time solution to the Ricci flow on Finsler manifolds}
In this section, we will prove that there is a one-to-one correspondence between the solutions of the Ricci flow and the Ricci-DeTurck flow on Finsler manifolds. Here, we recall some results which will be used in the sequel.
\begin{Lemma} \label{main9}
\cite[p.\ 82]{Chow1} Let $\{X_{t}:0\leq t<T\leq \infty\}$ be a continuous time-dependent family of vector fields on a compact manifold $M$. There exists a one-parameter family of diffeomorphisms $\{\varphi_{t}:M \longrightarrow M;\quad 0 \leq t<T \leq \infty\}$ defined on the same time interval such that
\begin{equation*}
\left\{
\begin{array}{l}
\frac{\partial}{\partial t} \varphi_{t}(x)=X_{t}[\varphi_{t}(x)],\cr
\varphi_{0}(x)=x,
\end{array}
\right.
\end{equation*}
for all $x\in M$ and $t\in[0,T)$.
\end{Lemma}
\begin{rem}\label{remark1}
Let $M$ be a compact Finsler manifold. According to Lemma \ref{main9}, there exists a unique one-parameter family of diffeomorphisms $\tilde{\varphi}_{t}$ on $SM$, such that
\begin{equation*}
\left\{
\begin{array}{l}
\frac{\partial}{\partial t}\varphi_{t}(z)=\xi(\varphi_{t}(z),t),\cr
\varphi_{0}=Id_{SM},
\end{array}
\right.
\end{equation*}
where, $z=(x,[y])\in SM$ and $t\in[0,T)$.
\end{rem}
\begin{rem}\label{remark2}
Let $\tilde{g}_{pq}$ be a solution to the Ricci-DeTurck flow and $\varphi_{t}$ the one-parameter global group of diffeomorphisms generating the vector field $\xi$. Since $\xi$ is a vector field on $SM$, then  $\varphi_{t}$ are homogeneous of degree zero. The zero-homogeneity of $\tilde{g}_{pq}$ implies that $\varphi_{t}^{*}(\tilde{g}_{pq})$ is also homogeneous of degree zero. That is,
\begin{equation*}
(\varphi_{t}^{*}\tilde{g}_{pq})(x,\lambda y)=\tilde{g}_{pq}(\varphi_{t}(x,\lambda y))=\tilde{g}_{pq}(\varphi_{t}(x,y))=
(\varphi_{t}^{*}\tilde{g}_{pq})(x,y).
\end{equation*}
Using the fact that $\tilde{g}_{pq}$ is positive definite and $\varphi_{t}^{*}$ are diffeomorphisms, $\varphi_{t}^{*}(\tilde{g}_{pq})$ is also positive definite. As well $\varphi_{t}^{*}(\tilde{g}_{pq})$ is symmetric. More intuitively,
\begin{equation*}
(\varphi_{t}^{*}\tilde{g})(X,Y)=g(\varphi_{t_{_*}}(X),\varphi_{t_{_*}}(Y))=g(\varphi_{t_{_*}}(Y),\varphi_{t_{_*}}(X))=(\varphi_{t}^{*}\tilde{g})(Y,X).
\end{equation*}
Therefore, $\varphi_{t}^{*}(\tilde{g}_{pq})$ determines a Finsler structure as follows
\begin{equation*}
F^{2}:=g_{pq}\tilde{y}^{p}\tilde{y}^{q},
\end{equation*}
where, $g_{pq}:=\varphi_{t}^{*}(\tilde{g}_{pq})$ and $\varphi_{t}^{*}y^p:=\tilde{y}^{p}$.
\end{rem}
\begin{lem}
Let $\varphi_{t}$ be a global one-parameter group of diffeomorphisms corresponding to the vector field $\xi$ and $(\gamma^{i}_{jk})_{\tilde{g}}$ and $(G^{i})_{\tilde{g}}$ are the Christoffel symbols and spray coefficients related to the Finsler metric $\tilde{g}$, respectively. Then we have
\begin{align}
&\varphi_{t}^{*}((\gamma_{jk}^{i})_{\tilde{g}})=(\gamma_{jk}^{i})_{\varphi_{t}^{*}(\tilde{g})},\label{Christoffel}\\
&\varphi_{t}^{*}(G^{i}_{\tilde{g}})=G^{i}_{\varphi_{t}^{*}(\tilde{g})},\label{spray}
\end{align}
where, $(\gamma_{jk}^{i})_{\tilde{g}}=\tilde{g}^{is}\frac{1}{2}(\frac{\partial \tilde{g}_{sj}}{\partial x^k}-\frac{\partial \tilde{g}_{jk}}{\partial x^s}+\frac{\partial \tilde{g}_{ks}}{\partial x^j})$ and $G^{i}_{\tilde{g}}=\frac{1}{2}(\gamma^{i}_{jk})_{\tilde{g}}y^jy^k$.
\end{lem}
\begin{proof}
Let us denote $\varphi_{t}^{*}x^i:=\tilde{x}^i$ and $\varphi_{t}^{*}y^i:=\tilde{y}^i$. By definition, we have
\begin{align*}
\varphi_{t}^{*}((\gamma^{i}_{jk})_{\tilde{g}})&=\varphi_{t}^{*}(\tilde{g}^{is}\frac{1}{2}(\frac{\partial \tilde{g}_{sj}}{\partial x^k}-\frac{\partial \tilde{g}_{jk}}{\partial x^s}+\frac{\partial \tilde{g}_{ks}}{\partial x^j}))\\
&=\varphi_{t}^{*}(\tilde{g}^{is})\varphi_{t}^{*}(\frac{1}{2}(\frac{\partial \tilde{g}_{sj}}{\partial x^k}-\frac{\partial \tilde{g}_{jk}}{\partial x^s}+\frac{\partial \tilde{g}_{ks}}{\partial x^j}))\\
&=\varphi_{t}^{*}(\tilde{g}^{is})\frac{1}{2}(\frac{\partial\varphi_{t}^{*}( \tilde{g}_{sj})}{\partial \tilde{x}^k}-\frac{\partial\varphi_{t}^{*}( \tilde{g}_{jk})}{\partial \tilde{x}^s}+\frac{\partial\varphi_{t}^{*}(\tilde{g}_{ks})}{\partial \tilde{x}^j})\\
&=(\gamma_{jk}^{i})_{\varphi_{t}^{*}(\tilde{g})}.
\end{align*}
Next, by means of (\ref{Christoffel}) we have
\begin{align*}
\varphi_{t}^{*}(G^{i}_{\tilde{g}})&=\varphi_{t}^{*}(\frac{1}{2}(\gamma^{i}_{jk})_{\tilde{g}}y^jy^k)=\frac{1}{2}\varphi_{t}^{*}((\gamma^{i}_{jk})_{\tilde{g}})\varphi_{t}^{*}y^j\varphi_{t}^{*}y^k\\
&=\frac{1}{2}(\gamma^{i}_{jk})_{\varphi_{t}^{*}(\tilde{g})}\tilde{y}^{j}\tilde{y}^{k}=G^{i}_{\varphi_{t}^{*}(\tilde{g})}.
\end{align*}
This completes the proof.
\end{proof}
\begin{lem}\label{lemmohem}
Let $\varphi_{t}$ be a global one-parameter group of diffeomorphisms generating the vector field $\xi$ and $\mathcal{R}ic_{\tilde{g}}$ the Ricci scalar related to the Finsler metric $\tilde{g}$, then we have
\begin{equation*}
\varphi_{t}^{*}(\mathcal{R}ic_{\tilde{g}})=\mathcal{R}ic_{\varphi_{t}^{*}(\tilde{g})}.
\end{equation*}
\end{lem}
\begin{proof}
Let us consider the \emph{reduced $hh$-curvature tensor} $R^{i}_{\,\,k}$ which is expressed entirely in terms of the $x$ and $y$ derivatives of the spray coefficients $G^{i}_{\tilde{g}}$.
\begin{equation*}
(R^{i}_{\,\,k})_{\tilde{g}}:=\frac{1}{\tilde{F}^2}(2\frac{\partial G^{i}_{\tilde{g}}}{\partial x^{k}}-\frac{\partial^{2}G^{i}_{\tilde{g}}}{\partial x^j\partial y^k}y^{j}+2G^{j}_{\tilde{g}}\frac{\partial^{2}G^{i}_{\tilde{g}}}{\partial y^j\partial y^k}-\frac{\partial G^{i}_{\tilde{g}}}{\partial y^{j}}\frac{\partial G^{j}_{\tilde{g}}}{\partial y^{k}}).
\end{equation*}
Therefore, we have
\begin{align*}
\varphi_{t}^{*}((R^{i}_{\,\,k})_{\tilde{g}})&=\varphi_{t}^{*}
(\frac{1}{\tilde{F}^2}(2\frac{\partial G^{i}_{\tilde{g}}}{\partial x^{k}}-\frac{\partial^{2}G^{i}_{\tilde{g}}}{\partial x^j\partial y^k}y^{j}+2G^{j}_{\tilde{g}}\frac{\partial^{2}G^{i}_{\tilde{g}}}{\partial y^j\partial y^k}-\frac{\partial G^{i}_{\tilde{g}}}{\partial y^{j}}\frac{\partial G^{j}_{\tilde{g}}}{\partial y^{k}}))\nonumber\\
&=\varphi_{t}^{*}(\frac{1}{\tilde{F}^2})\varphi_{t}^{*}(2\frac{\partial G^{i}_{\tilde{g}}}{\partial x^{k}}-\frac{\partial^{2}G^{i}_{\tilde{g}}}{\partial x^j\partial y^k}y^{j}+2G^{j}_{\tilde{g}}\frac{\partial^{2}G^{i}_{\tilde{g}}}{\partial y^j\partial y^k}-\frac{\partial G^{i}_{\tilde{g}}}{\partial y^{j}}\frac{\partial G^{j}_{\tilde{g}}}{\partial y^{k}}).
\end{align*}
Thus, we get
\begin{align*}
\varphi_{t}^{*}((R^{i}_{\,\,k})_{\tilde{g}})=&\frac{1}{\varphi_{t}^{*}(\tilde{F}^2)}(2\frac{\partial(\varphi_{t}^{*} (G^{i}_{\tilde{g}}))}{\partial \tilde{x}^{k}}-\frac{\partial^{2}(\varphi_{t}^{*}(G^{i}_{\tilde{g}}))}{\partial \tilde{x}^j\partial \tilde{y}^k}\tilde{y}^{j}\nonumber\\
&+2\varphi_{t}^{*}(G^{j}_{\tilde{g}})\frac{\partial^{2}(\varphi_{t}^{*}(G^{i}_{\tilde{g}}))}{\partial \tilde{y}^j\partial \tilde{y}^k}-\frac{\partial(\varphi_{t}^{*}( G^{i}_{\tilde{g}}))}{\partial \tilde{y}^{j}}\frac{\partial (\varphi_{t}^{*}(G^{j}_{\tilde{g}}))}{\partial \tilde{y}^{k}}).
\end{align*}
Letting $i=k$ in the last equation and using (\ref{spray})  implies
\begin{equation*}
\varphi_{t}^{*}(\mathcal{R}ic_{\tilde{g}})=\mathcal{R}ic_{\varphi_{t}^{*}(\tilde{g})},
\end{equation*}
as we have claimed.
\end{proof}
Now we are in a position to prove the following proposition.
\begin{prop} \label{main11}
Fix a compact Finsler manifold $(M,\bar{F})$ with the related Finsler metric tensor $h$. Let $\tilde{F}(t)$ be a family of solutions to the Ricci-DeTurck flow
\begin{eqnarray} \label{46}
\frac{\partial}{\partial t}\tilde{F}^{2}(t)=-2\tilde{F}^{2}(t)\mathcal{R}ic(\tilde{g}(t))-\mathcal{L}_{\xi}\tilde{F}^{2}(t),
\end{eqnarray}
where, $\xi=(\Phi_{\tilde{g}(t),h}Id)^{i}\frac{\partial}{\partial x^i}$ and $t\in[0,T)$. Moreover, let $\varphi_{t}$ be a one-parameter family of diffeomorphisms satisfying
\begin{eqnarray*} 
\frac{\partial}{\partial t}\varphi_{t}(z)=\xi(\varphi_{t}(z),t),\nonumber
\end{eqnarray*}
for $z\in SM$ and $t\in[0,T)$. Then the Finsler structures $F(t)$ form a solution to the Finslerian Ricci flow (\ref{20}) where, $F(t)$ is defined by
\begin{eqnarray*} 
F^{2}(t):=g_{pq}\tilde{y}^{p}\tilde{y}^{q}=\varphi_{t}^{*}(\tilde{F}^{2}(t)),\nonumber
\end{eqnarray*}
where, $g_{pq}:=\varphi_{t}^{*}(\tilde{g}_{pq})$ and $\varphi_{t}^{*}y^p:=\tilde{y}^{p}$.
\end{prop}
\begin{proof}
In order to show that $F(t)$ forms a solution to the Finslerian Ricci flow (\ref{20}) we need to show that  $\frac{\partial}{\partial t}(\log F(t))=-\mathcal{R}ic.$
Derivation of $F^{2}(t)=\varphi_{t}^{*}(\tilde{F}^{2}(t))$ with respect to the parameter $t$, leads to
\begin{equation} \label{50}
\frac{\partial}{\partial t}(\log F(t))=\frac{1}{2}\frac{\frac{\partial}{\partial t}(\varphi_{t}^{*}(\tilde{F}^{2}(t)))}{\varphi_{t}^{*}(\tilde{F}^{2}(t))}.
\end{equation}
The term $\frac{\partial}{\partial t}(\varphi_{t}^{*}\tilde{F}^{2}(t))$ becomes
\begin{eqnarray} \label{51}
\frac{\partial}{\partial t}(\varphi_{t}^{*}\tilde{F}^{2}(t))&=&\frac{\partial}{\partial s}(\varphi_{s+t}^{*}(\tilde{F}^{2}(s+t)))\mid_{s=0}\\
&=&\varphi_{t}^{*}(\frac{\partial}{\partial t}\tilde{F}^{2}(t))+\frac{\partial}{\partial s}(\varphi_{s+t}^{*}(\tilde{F}^{2}(t)))\mid_{s=0}\nonumber \\
&=&\varphi_{t}^{*}(\frac{\partial}{\partial t}\tilde{F}^{2}(t))+\frac{\partial}{\partial s}((\varphi_{t}^{-1}\circ\hspace{0.1cm}\varphi_{t+s})^{*}(\varphi_{t}^{*}(\tilde{F}^{2}(t))))\mid_{s=0}\nonumber\\
&=&\varphi_{t}^{*}(\frac{\partial}{\partial t}\tilde{F}^{2}(t))+\mathcal{L}_{\frac{\partial}{\partial s}(\varphi_{t}^{-1}\circ\hspace{0.1cm}\varphi_{t+s})\mid_{s=0}}\varphi_{t}^{*}(\tilde{F}^{2}(t)).\nonumber
\end{eqnarray}
On the other hand, we have
\begin{eqnarray*} 
\frac{\partial}{\partial s}(\varphi_{t}^{-1}\circ\hspace{0.1cm}\varphi_{t+s})\mid_{s=0}=(\varphi_{t}^{-1})_{*}((\frac{\partial}{\partial s}\varphi_{s+t})\mid_{s=0})=(\varphi_{t}^{-1})_{*}(\xi).
\end{eqnarray*}
Hence, (\ref{51}) is written
\begin{eqnarray*}
\frac{\partial}{\partial t}(\varphi_{t}^{*}\tilde{F}^{2}(t))=\varphi_{t}^{*}(\frac{\partial}{\partial t}\tilde{F}^{2}(t))+\mathcal{L}_{(\varphi_{t}^{-1})_{*}(\xi)}\varphi_{t}^{*}(\tilde{F}^{2}(t)).
\end{eqnarray*}
Replacing the last equation in (\ref{50}) and using the assumption (\ref{46}) we get
\begin{eqnarray*}
\frac{\partial}{\partial t}(\log F(t))&=&\frac{1}{2}\frac{\varphi_{t}^{*}(\frac{\partial}{\partial t}\tilde{F}^{2}(t))+\mathcal{L}_{(\varphi_{t}^{-1})_{*}(\xi)}
\varphi_{t}^{*}(\tilde{F}^{2}(t))}{\varphi_{t}^{*}(\tilde{F}^{2}(t))} \\
&=&\frac{1}{2}\frac{\varphi_{t}^{*}(-2\tilde{F}^{2}(t)\mathcal{R}ic(\tilde{F}(t))-\mathcal{L}_{\xi}\tilde{F}^{2}(t))+
\mathcal{L}_{(\varphi_{t}^{-1})_{*}(\xi)}\varphi_{t}^{*}(\tilde{F}^{2}(t))}{\varphi_{t}^{*}(\tilde{F}^{2}(t))}\\
&=&\frac{1}{2}\frac{\varphi_{t}^{*}(-2\tilde{F}^{2}(t)\mathcal{R}ic(\tilde{F}(t)))-\varphi_{t}^{*}(\mathcal{L}_{\xi}\tilde{F}^{2}(t)))
+\mathcal{L}_{(\varphi_{t}^{-1})_{*}(\xi)}\varphi_{t}^{*}(\tilde{F}^{2}(t))}{\varphi_{t}^{*}(\tilde{F}^{2}(t))}\\
&=&\frac{1}{2}\frac{-2\varphi_{t}^{*}(\tilde{F}^{2}(t))\varphi_{t}^{*}(\mathcal{R}ic(\tilde{F}(t)))}{\varphi_{t}^{*}(\tilde{F}^{2}(t))}.
\end{eqnarray*}
By virtue of Lemma \ref{lemmohem} we have
\begin{eqnarray*}
\frac{\partial}{\partial t}(\log F(t))
=-\varphi_{t}^{*}(\mathcal{R}ic(\tilde{F}(t)))
=-\mathcal{R}ic_{\varphi_{t}^{*}(\tilde{F}(t))}
=-\mathcal{R}ic_{F(t)}.
\end{eqnarray*}
Therefore, the Finsler structures $F(t)$ form a solution to the Finslerian Ricci flow. Hence the proof is complete.
\end{proof}
{\bf Proof of Theorem \ref{main14}.}   We check the existence of a solution to the Finslerian Ricci flow.   Recall that Theorem \ref{main8} asserts that, there exists a solution $\tilde{F}(t)$ to the Finslerian Ricci-DeTurck flow (\ref{22}) which is defined on some time interval $[0,T)$ and satisfies $\tilde{F}(0)=F_{0}$. Let $\varphi_{t}$ be the solution of the ODE
\begin{eqnarray*}
\frac{\partial}{\partial t}\varphi_{t}(z)=(\Phi_{\tilde{g}(t),h}Id)(\varphi_{t}(z),t)=\xi(\varphi_{t}(z),t),\nonumber
\end{eqnarray*}
with the initial condition $\varphi_{0}(z)=z$, for $z\in SM$ and $t\in[0,T)$. By Proposition \ref{main11}, the Finsler structures $F^{2}(t)=\varphi^{*}_{t}(\tilde{F}^{2}(t))$ form a solution to the Finslerian Ricci flow (\ref{20}) with $F(0)=F_{0}$. This completes the proof of Theorem \ref{main14}.\hspace{\stretch{1}}$\Box$

In the following proposition,  we assume that there exists a solution to the Finslerian Ricci flow, from which we construct a solution to the Ricci-DeTurck flow.
\begin{prop}\label{main12}
Consider  a fixed compact Finsler manifold $(M,\bar{F})$, with the related Finsler metric tensor $h$. Let $F(t)$, $t\in[0,T)$, be a family of solutions to the Ricci flow and $\varphi_{t}$ a one-parameter family of diffeomorphisms on $SM$ deforming under the following flow,
\begin{equation} \label{55}
\frac{\partial}{\partial t}\varphi_{t}=\Phi_{g(t),h}\varphi_{t}.\nonumber
\end{equation}
Then the Finsler structures $\tilde{F}(t)$ defined by $F^{2}(t)=\varphi^{*}_{t}(\tilde{F}^{2}(t))$ form a solution to the following Ricci-DeTurck flow
\begin{equation} \label{56}
\frac{\partial}{\partial t}\tilde{F}^{2}(t)=-2\tilde{F}^{2}(t)\mathcal{R}ic(\tilde{g}(t))-\mathcal{L}_{\xi}\tilde{F}^{2}(t),\nonumber
\end{equation}
where, $\xi=(\Phi_{\tilde{g}(t),h}Id)^{i}\frac{\partial}{\partial x^i}$ and $\tilde{g}(t)$ is the metric tensor related to $\tilde{F}(t)$. Furthermore, for all $z\in SM$ and $t\in[0,T)$ we have
\begin{equation*}
\frac{\partial}{\partial t}\varphi_{t}(z)=\xi(\varphi_{t}(z),t).
\end{equation*}
\end{prop}
\begin{proof}
Using Lemma \ref{main2} we have
\begin{align*}
\frac{\partial}{\partial t}\varphi_{t}&=\Phi_{g(t),h}\varphi_{t}=\Phi_{\varphi^{*}_{t}
(\tilde{g}(t)),h}\varphi_{t}=\Phi_{\varphi^{*}_{t}
(\tilde{g}(t)),h} Id \circ\varphi_{t}\\
&=\Phi_{\varphi^{*}_{t}
(\tilde{g}(t)),h}\varphi_{t}^{*} Id=\Phi_{\tilde{g}(t),h}Id=\xi,
\end{align*}
for all $z\in SM$ and $t\in[0,T)$. Having $F^{2}(t)=\varphi^{*}_{t}(\tilde{F}^{2}(t))$, we obtain
\begin{align}\label{59+1}
\frac{\partial}{\partial t}(\log F(t))&=
\frac{1}{2}\frac{\frac{\partial}{\partial t}(\varphi^{*}_{t}(\tilde{F}^{2}(t)))}{\varphi^{*}_{t}(\tilde{F}^{2}(t))}\nonumber\\
&=\frac{1}{2}\frac{\varphi^{*}_{t}(\frac{\partial}{\partial t}\tilde{F}^{2}(t))+\mathcal{L}_{(\varphi^{-1}_{t})_{*}(\xi)}
\varphi^{*}_{t}(\tilde{F}^{2}(t))}{\varphi^{*}_{t}(\tilde{F}^{2}(t))}\nonumber\\
&=\frac{1}{2}\frac{\varphi^{*}_{t}(\frac{\partial}{\partial t}\tilde{F}^{2}(t)+\mathcal{L}_{\xi}\tilde{F}^{2}(t))}{\varphi^{*}_{t}(\tilde{F}^{2}(t))}.
\end{align}
By assumption, $F(t)$ form a solution to the Finslerian Ricci flow (\ref{20})
\begin{equation} \label{60}
0=\frac{\partial}{\partial t}(\log F(t))+\mathcal{R}ic_{F(t)}.
\end{equation}
By means of (\ref{59+1}), (\ref{60}) and Lemma \ref{lemmohem} we have
\begin{align*}
0&=\frac{\varphi^{*}_{t}(\frac{\partial}{\partial t}\tilde{F}^{2}(t)+\mathcal{L}_{\xi}\tilde{F}^{2}(t))}{\varphi^{*}_{t}(\tilde{F}^{2}(t))}+2\mathcal{R}ic_{F(t)}\nonumber\\
&=\frac{\varphi^{*}_{t}(\frac{\partial}{\partial t}\tilde{F}^{2}(t)+\mathcal{L}_{\xi}\tilde{F}^{2}(t))}{\varphi^{*}_{t}(\tilde{F}^{2}(t))}
+2\mathcal{R}ic_{\varphi^{*}_{t}(\tilde{F}(t))}\nonumber\\
&=\frac{\varphi^{*}_{t}(\frac{\partial}{\partial t}\tilde{F}^{2}(t)+\mathcal{L}_{\xi}\tilde{F}^{2}(t))}{\varphi^{*}_{t}(\tilde{F}^{2}(t))}
+2\varphi^{*}_{t}(\mathcal{R}ic_{\tilde{F}(t)})\nonumber\\
&=\frac{\varphi^{*}_{t}(\frac{\partial}{\partial t}\tilde{F}^{2}(t)+\mathcal{L}_{\xi}\tilde{F}^{2}(t))+2
\varphi^{*}_{t}(\tilde{F}^{2}(t))\varphi^{*}_{t}(\mathcal{R}ic_{\tilde{F}(t)})}{\varphi^{*}_{t}
(\tilde{F}^{2}(t))}\nonumber\\
&=\frac{\varphi^{*}_{t}(\frac{\partial}{\partial t}\tilde{F}^{2}(t)+\mathcal{L}_{\xi}\tilde{F}^{2}(t)+2
\tilde{F}^{2}(t)\mathcal{R}ic_{\tilde{F}(t)})}{\varphi^{*}_{t}
(\tilde{F}^{2}(t))}\nonumber.
\end{align*}
Therefore,
$
\varphi^{*}_{t}(\frac{\partial}{\partial t}\tilde{F}^{2}(t)+\mathcal{L}_{\xi}\tilde{F}^{2}(t)+2
\tilde{F}^{2}(t)\mathcal{R}ic_{\tilde{F}(t)})=0$, which  implies
\begin{equation}
\frac{\partial}{\partial t}\tilde{F}^{2}(t)=-2
\tilde{F}^{2}(t)\mathcal{R}ic_{\tilde{F}(t)}-\mathcal{L}_{\xi}\tilde{F}^{2}(t).\nonumber
\end{equation}
Hence, $\tilde{F}(t)$ is a solution to the Ricci-DeTurck flow, as we have stated.
\end{proof}
{\bf Proof of Theorem \ref{isotropic}.} The existence statement has been proved in Theorem \ref{main14} for a general Finsler structure. For uniqueness statement in the isotropic case, assume that $F_{1}(t)$ and $F_{2}(t)$ are both solutions to the Finslerian Ricci flow defined on some time interval $[0,T)$ and satisfy $F_{1}(0)=F_{2}(0)$ at the point $t=0$. We claim $F_{1}(t)=F_{2}(t)$ for all $t\in[0,T)$. In order to prove this fact, we argue by contradiction. Suppose that $F_{1}(t)\neq F_{2}(t)$ for some $t\in[0,T)$. Let us consider a real number $\tau\in[0,T)$ where $\tau=\inf\{t\in[0,T):F_{1}(t)\neq F_{2}(t)\}$. We will easily see, $F_{1}(\tau)=F_{2}(\tau)$. Let $\varphi^{1}_{t}$ be a solution of the flow
\begin{equation} \label{67}
\frac{\partial}{\partial t}\varphi^{1}_{t}=\Phi_{g_{1}(t),h}\varphi^{1}_{t},\nonumber
\end{equation}
with the initial condition $\varphi^{1}_{\tau}=Id$ and $\varphi^{2}_{t}$ a solution of the flow
\begin{equation}
\frac{\partial}{\partial t}\varphi^{2}_{t}=\Phi_{g_{2}(t),h}\varphi^{2}_{t},\nonumber
\end{equation}
with the initial condition $\varphi^{2}_{\tau}=Id$. It follows from the standard theory of parabolic differential equations that $\varphi^{1}_{t}$ and $\varphi^{2}_{t}$ are defined on some time interval $[\tau,\tau+\epsilon)$, where, $\epsilon$ is a positive real number. Moreover, if we choose $\epsilon>0$ small enough, then $\varphi^{1}_{t}$ and $\varphi^{2}_{t}$ are diffeomorphisms for all $t\in[\tau,\tau+\epsilon)$. For each $t\in[\tau,\tau+\epsilon)$ we define the two Finsler structures $\tilde{F}_{1}(t)$ and $\tilde{F}_{2}(t)$ by $(F_{1}(t))^{2}=(\varphi^{1}_{t})^{*}(\tilde{F}_{1}(t))^{2}$ and $(F_{2}(t))^{2}=(\varphi^{2}_{t})^{*}(\tilde{F}_{2}(t))^{2}$. It follows from Proposition \ref{main12} that $\tilde{F}_{1}(t)$ and $\tilde{F}_{2}(t)$ are the solutions of the Finslerian Ricci-DeTurck flow. Since $\tilde{F}_{1}(\tau)=\tilde{F}_{2}(\tau)$, the uniqueness of solution to the Finslerian Ricci-DeTurck flow on isotropic Finsler manifolds mentioned in Remark \ref{remarkunique} implies that $\tilde{F}_{1}(t)=\tilde{F}_{2}(t)$ for all $t\in[\tau,\tau+\epsilon)$. For each $t\in[\tau,\tau+\epsilon)$, we define a vector field $\xi$ on $SM$ by
\begin{equation}
\xi=\Phi_{\tilde{g}_{1}(t),h}Id=\Phi_{\tilde{g}_{2}(t),h}Id.\nonumber
\end{equation}
By Proposition \ref{main12}, we have
\begin{eqnarray*}
\frac{\partial}{\partial t}\varphi^{1}_{t}(z)=\xi(\varphi^{1}_{t}(z),t),\nonumber
\end{eqnarray*}
and
\begin{eqnarray*}
\frac{\partial}{\partial t}\varphi^{2}_{t}(z)=\xi(\varphi^{2}_{t}(z),t),\nonumber
\end{eqnarray*}
for $z\in SM$ and $t\in[\tau,\tau+\epsilon)$. Since $\varphi^{1}_{\tau}=\varphi^{2}_{\tau}=Id$, it follows that $\varphi^{1}_{t}=\varphi^{2}_{t}$ for all $t\in[\tau,\tau+\epsilon)$.  Putting these facts together, we conclude that
\begin{equation*}
(F_{1}(t))^{2}=(\varphi^{1}_{t})^{*}(\tilde{F}_{1}(t))^{2}=(\varphi^{2}_{t})^{*}(\tilde{F}_{2}(t))^{2}=
(F_{2}(t))^{2},
\end{equation*}
for all $t\in[\tau,\tau+\epsilon)$. Therefore, $F_{1}(t)=F_{2}(t)$ for all $t\in[\tau,\tau+\epsilon)$. This contradicts the definition of $\tau$. Thus the uniqueness holds well. This completes the proof of Theorem \ref{isotropic}.\hspace{\stretch{1}}$\Box$
\begin{ex}
Let $(M,F_0)$ be an Einstein Randers manifold with dimension $n>2$. We are going to obtain a solution to the Ricci flow \eqref{20}. Using Schur's lemma for Einstein Randers metrics \cite[p.\ 30]{Rob}, the Ricci scalar $\mathcal{R}ic$ of $F_0$ is necessarily constant, that is $\mathcal{R}ic_{F_0}=K$ where, $K$ is a constant. Consider a family of linear scalars $\tau(t)$ defined by
\begin{equation*}
\tau(t):=1-2Kt>0.
\end{equation*}
Define a smooth one-parameter family of Finsler structures on $M$ by
\begin{equation*}
F^{2}(t):=\tau(t)F_{0}^{2}.
\end{equation*}
Thus we have
\begin{equation*}
\log(F(t))=\frac{1}{2}\log(\tau(t)F_{0}^2).
\end{equation*}
Derivative with respect to $t$ yields
\begin{equation}\label{exp1}
\frac{\partial}{\partial t}\log(F(t))=-\frac{K}{\tau(t)}=-\frac{\mathcal{R}ic_{F_{0}}}{\tau(t)}.
\end{equation}
On the other hand, by straight forward computations we have $\frac{1}{\tau(t)}\mathcal{R}ic_{F_{0}}=\mathcal{R}ic_{\tau(t)^{\frac{1}{2}}F_{0}}$, for more details see \cite[p.\ 926]{BY}. Replacing the last relation in (\ref{exp1}) leads to
\begin{equation*}
\frac{\partial}{\partial t}\log(F(t))=-\mathcal{R}ic_{\tau(t)^{\frac{1}{2}}F_{0}}=-\mathcal{R}ic_{{F(t)}}.
\end{equation*}
Hence, $F(t)$ is a solution to the Ricci flow equation \eqref{20}.
\end{ex}
\begin{ex}
Let $(M,F_0)$ be a Finsler manifold and $V=V^{i}(x)\frac{\partial}{\partial x^i}$ a vector field on $M$. The triple $(M,F_0,V)$ is called a Finslerian Ricci soliton if there exists a constant $\lambda$ such that
\begin{equation}\label{soliton}
2Ric_{ij}+\mathcal{L}_{\hat{V}}g_{ij}=2\lambda g_{ij},
\end{equation}
where, $g_{ij}$ is the Hessian related to the Finsler structure $F_0$, $\hat{V}$ the complete lift of $V$ and $\lambda\in\mathbb{R}$. By multiplying $y^iy^j$ in  both sides of \eqref{soliton}, we obtain
\begin{equation*}
2F_{0}^{2}\mathcal{R}ic+\mathcal{L}_{\hat{V}}F_{0}^{2}=2\lambda F_{0}^{2}.
\end{equation*}
Depending on the sign of $\lambda$, a Finslerian Ricci soliton is called shrinking $(\lambda>0)$, steady $(\lambda=0)$, or expanding $(\lambda<0)$. Suppose $(M,F_0,V)$ is a compact Finslerian Ricci soliton. For each point $p\in M$, we denote by $\varphi_{t}(p)$ the unique solution of the ordinary differential equation (ODE)
\begin{equation*}
\frac{\partial}{\partial t}\varphi_{t}(p)=\frac{1}{1-2\lambda t}V|_{\varphi_{t}(p)},
\end{equation*}
with the initial condition $\varphi_{0}(p)=p$. This defines a one-parameter family of diffeomorphisms on $M$. One can define a smooth one-parameter family of Finsler structures on $TM$ as follow
\begin{equation*}
F^{2}(t)=(1-2\lambda t)\tilde{\varphi}^{*}_{t}(F_{0}^{2}),
\end{equation*}
where, $\tilde{\varphi}^{*}_{t}$ is the one-parameter group on $TM$ associated to the complete lift $\frac{1}{1-2\lambda t}\hat{V}$. Then $F(t)$ form a solution to the Ricci flow equation \eqref{20}, see \cite[p.\ 928]{BY}.
\end{ex}
In conclusion, this paper has studied the short-time existence and uniqueness of Ricci flow solutions on isotropic Finsler manifolds with dimension $n$. We have shown that the short-time existence of the solution can be established, and if the Finsler manifold is isotropic, the solution is unique. The results of this paper highlight the versatility of the Ricci flow, as it can be applied to various areas of geometry and physics,  as well as computer graphics, medical imaging, and surface matching.
 
  Behroz Bidabad$^{1,2}$,  Maral Khadem Sedaghat$^{1}$,\\
 bidabad@aut.ac.ir; m\_sedaghat@aut.ac.ir \\
  1. Department of Mathematics and Computer Science, \\ Amirkabir University of Technology (Tehran Polytechnic),\\ Hafez Ave., 15914 Tehran, Iran.\\
2. Institut de Mathematique de Toulouse (IMT),
Universite' Paul Sabatier, 118 route de Narbonne - F-31062 Toulouse, France.\\
behroz.bidabad@math.univ-toulouse.fr

\end{document}